\documentclass[a4paper,11pt]{article}
\usepackage[latin1]{inputenc}
\usepackage[T1]{fontenc}
\usepackage{array}
\usepackage{fancyhdr}
\usepackage{longtable}
\usepackage{tabularx}
\usepackage{rotating}

\usepackage{makeidx}
\usepackage{amsmath}
\usepackage{empheq}
\usepackage{amsthm}
\usepackage{amssymb}
\usepackage[mathscr]{eucal}
\usepackage{enumerate}
\usepackage[dvips]{epsfig}
\usepackage{float} 
\usepackage{ifthen} 
\usepackage{extarrows} 
\usepackage{stmaryrd}  
\usepackage{arydshln}    
\usepackage{multirow}
\usepackage{mathtools}
\usepackage[dvips]{geometry}
\usepackage{bbm}
\usepackage{xcolor}

\usepackage{listings}
\usepackage{subfig}
\usepackage{dsfont}
\usepackage{xspace}
\usepackage{xfrac}

\usepackage{todonotes}
\usepackage{changes} 

\providecommand{\noopsort}[1]{}

\theoremstyle{plain}
\newtheorem{thm}{Theorem}[section]
\newtheorem{prop}[thm]{Proposition}
\newtheorem{lem}[thm]{Lemma}

\theoremstyle{definition}

\theoremstyle{remark}
\newtheorem{remark}{Remark}[section]

\newcommand{\Cov}{\operatorname{Cov}}

\newcommand{\Prob}{\mathrm{P}}

\newcommand{\ind}{\mathbbmss{1}}

\newcommand{\Span}{\operatorname{span}}

\newcommand{\DomX}{D}

\usepackage[bookmarksopen,bookmarksnumbered,colorlinks=true,linkcolor=red,citecolor=blue]{hyperref}
\usepackage{bookmark}
\hypersetup{
	pdfauthor={author},
	% ...
}

\newcommand{\YDFN}{Y}	
\newcommand{\YWP}{Y^{WP}}	

\newcommand{\ESRK}{\operatorname{ERKM}}
\newcommand{\SESRK}{\operatorname{DFMM}}

\newcommand{\LIE}{\operatorname{LIE}}

\newcommand{\RKTN}{\operatorname{ERKM1.5}}
\newcommand{\RKTNV}{\operatorname{ERKM1.5V1}}
\newcommand{\EWP}{\operatorname{EWP}}
\newcommand{\EXE}{\operatorname{EXE}}

\newcommand{\e}{\mathrm{e}}
\newcommand{\dd}{\, \mathrm{d}}

\newcommand{\norm}[1]{\left\lVert#1\right\rVert}
\newcommand{\normA}[2]{#1\lVert #2 #1\rVert}

\newcommand{\normLpR}[1]{\left\lVert#1\right\rVert_{L^p(\Omega;\mathbb{R})}}
\newcommand{\normALpR}[2]{#1\lVert #2 #1\rVert_{L^p(\Omega;\mathbb{R})}}

\newcommand{\normLpH}[3]{\left\lVert#3\right\rVert_{L^{#1}(\Omega;H_{#2})}}
\newcommand{\normALpH}[4]{#1\lVert #4 #1\rVert_{L^{#2}(\Omega;H_{#3})}}
\newcommand{\normLpHS}[3]{\left\lVert #3 \right\rVert_{L^{#1}(\Omega;HS(U_0,H_{#2}))}}
\newcommand{\normALpHS}[4]{#1\lVert #4 #1\rVert_{L^{#2}(\Omega;HS(U_0,H_{#3}))}}

\newcommand{\E}[1]{\mathbb{E}\left[ #1 \right]}

\newcommand{\Wds}{\bigg( \int_{lh}^{(l+1)h} ( W_s - W_{lh} ) \dd s \bigg)}
\newcommand{\WdsOK}{\int_{lh}^{(l+1)h} \left( W_s - W_{lh} \right) \dd s}

\newcommand{\etaA}[1]{\big\lVert (\eta-A)^{#1}\big\rVert_{L(H)}}
\newcommand{\etaAe}[1]{\Big\lVert (\eta-A)^{#1} \e^{\tfrac{A}{2}(m-l-\tfrac{1}{2})h} \Big\rVert_{L(H)}}

\newcommand{\sumJ}{\sum_{j\in\mathcal{J}}}

\newcommand{\cdrei}{\hat{c}_1}
\newcommand{\cvier}{\hat{c}_2}
\newcommand{\cfuenf}{\hat{c}_3}
\newcommand{\cacht}{\hat{c}_4}
\newcommand{\cneun}{\hat{c}_5}
\newcommand{\czehn}{\hat{c}_6}
\newcommand{\celf}{\hat{c}_7}
\newcommand{\aelf}{\hat{c}_8}

\allowdisplaybreaks

\title{An Exponential Stochastic Runge-Kutta Type Method of Order up to 1.5 for SPDEs
of Nemytskii-type}

\author{Claudine von Hallern$^1$%
\thanks{
e-mail: claudine.von.hallern@uni-hamburg.de},
\ \ 
Ricarda Mi{\ss}feldt$^2$%
\ \
and Andreas R\"o{\ss}ler$^2$\thanks{e-mail: roessler@math.uni-luebeck.de 
\\ Date: December 10th, 2024
}
\bigskip
\\
\small{$^1$Department of Mathematics, Universit\"at Hamburg,} \\
\small{Grindelberg 5, 20144 Hamburg, Germany} \\[0.2cm]
\small{$^2$Institute of Mathematics, Universit\"at zu L\"ubeck,} \\
\small{Ratzeburger Allee 160, 23562 L\"ubeck, Germany} 
}

\date{}

\begin{document}

\maketitle

\begin{abstract}
\noindent
For the approximation of solutions for stochastic partial differential equations, 
numerical methods that obtain a high order of convergence and at the same time 
involve reasonable computational cost are of particular interest. We therefore 
propose a new numerical method of exponential stochastic Runge-Kutta type that allows 
for convergence with a temporal order of up to~$\sfrac{3}{2}$ and that can be combined with 
several spatial discretizations. The developed family of derivative-free schemes 
is tailored to stochastic partial 
differential equations of Nemytskii-type, i.e., with pointwise multiplicative noise
operators. We prove the strong convergence of these schemes in the root mean-square 
sense and present some numerical examples that reveal the theoretical results. 
\end{abstract}
%
%
%
%
\section{Introduction}
In various areas, such as neuroscience, biology, physics
and geosciences, stochastic partial differential equations (SPDEs)
are applied as models to describe dynamical evolution. Typical models
include, e.g., stochastic reaction-diffusion equations or stochastic wave equations.
The numerical approximation of SPDEs
is a major research field as an analytical solution for this type of equations
can only be computed in very few cases. The complexity of these equations also 
calls for numerical schemes that attain a high order of temporal convergence to 
decrease the overall computational cost. This is also the reason to introduce
derivative-free schemes since Taylor type schemes require the calculation of
derivatives for the involved operators which can be a rather hard task. 
In the last years, there has been increasing work on higher 
order schemes for strong approximation of solutions for semilinear SPDEs.
Note that numerical schemes converging in the strong sense are usually applied
if single trajectories of the solution process are the object of interest.
For example, numerical schemes of Milstein type 
were introduced and investigated, e.g., by Barth and Lang~\cite{MR2996432,MR3027891},
Jentzen and R\"ocker~\cite{MR3320928}, von~Hallern and R\"o{\ss}ler~\cite{MR4112639new}
and Reisinger and Wang~\cite{MR4032895}. 
Furthermore, the Wagner-Platen type scheme
proposed by Becker, Jentzen and Kloeden in~\cite{MR3534472} is a Taylor type
scheme and allows for an even higher order of convergence compared to Euler 
or Milstein type schemes.
Concerning derivative-free schemes, Wang and Gan~\cite{MR3011387} introduced 
a Milstein type scheme for equations containing pointwise multiplicative 
operators. More general equations, with and without a commutative noise 
condition, are also treated in, e.g., \cite{HalRoe2023pre,MR3842926}. 
In general, higher order approximation schemes incorporate iterated
stochastic integrals that need to be simulated, see, e.g., \cite{MR3949104}.
However, in case of commutative noise the terms belonging to iterated stochastic 
integrals can often be transformed into terms that only contain increments of the
driving $Q$-Wiener process that can be easily simulated. This situation naturally
arises in case of pointwise multiplicative noise operators as considered in
the following.
%

In contrast to the numerical analysis of finite dimensional stochastic differential
equations, we cannot state a fixed number for the order of convergence as this 
is dependent on the particular SPDE under consideration, e.g., the parameters determining
the regularity of the corresponding solution. However, we can still sort the 
schemes by the order of convergence that they can achieve under suitable assumptions. 
The Wagner-Platen type scheme~\cite{MR3534472} has the potential to obtain a 
higher order of convergence compared to the Euler and Milstein schemes mentioned above
provided that, e.g., the coefficients of the considered SPDE are sufficiently smooth.

Focusing on higher order Taylor type numerical schemes, the main disadvantage 
is that they incorporate derivatives of the coefficients. The calculation of 
derivatives in Hilbert spaces, especially for operators, can be difficult which 
may be a reason that prevents people from applying such schemes. This motivates
the development of numerical methods that are easier to implement keeping the
advantage of high orders of convergence.

In this work, we introduce a new exponential stochastic Runge-Kutta type method
that defines a whole family of exponential stochastic Runge-Kutta type schemes. These 
schemes obtain the same high temporal order of convergence as the Wagner-Platen type scheme 
but they do not need the calculation of any derivatives. Further, we show that the 
proposed exponential stochastic Runge-Kutta type schemes allow for significant 
savings of computational cost.
%

%
In the following, we are concerned with the problem of approximating solutions 
of semilinear stochastic equations that are of the form
\begin{equation} \label{SPDE-NemType}
	\begin{split}
	\mathrm{d} X_t(x) &= \big( A X_t(x) + f(x,X_t(x)) \big) \, \mathrm{d}t 
	+ b(x,X_t(x)) \, \mathrm{d}W_t(x), \quad t\in(0,T], \\
	X_0(x) &= \xi(x) ,
	\end{split}
\end{equation}
for $x \in \DomX$ with some domain $\DomX \subseteq \mathbb{R}^n$.
Here, $A$ is a linear operator, $f$ and $b$
are some suitable mappings that may be nonlinear, $b$ maps into the
space of pointwise multiplicative operators and $(W_t)_{t \in [0,T]}$ 
denotes a $Q$-Wiener process.
Equation~\eqref{SPDE-NemType} can be reformulated as 
an abstract stochastic evolution equation of the form
\begin{equation} \label{SPDE}
	\begin{split}
  	\mathrm{d} X_t &= \big( AX_t+F(X_t)\big) \dd t + B(X_t) \dd W_t, 
  	\quad t \in (0,T], \\
  	X_0 &= \xi,
	\end{split}
\end{equation}
with corresponding mappings $F$ and $B$. In the case that $A$ denotes
a partial differential operator, equations~\eqref{SPDE-NemType} and~\eqref{SPDE}
may also represent some SPDE.
Details and assumptions on the involved mappings, operators and processes are 
given in Section~\ref{Sec:Setting}.
The goal is to approximate the mild solution to 
stochastic evolution equations of type~\eqref{SPDE-NemType} in the strong sense.

In the next section, 
we introduce an exponential stochastic Runge-Kutta type 
method for the strong approximation of mild solutions to stochastic evolution 
equations of type~\eqref{SPDE-NemType}. The proposed method is based on 
the exponential Wagner-Platen type scheme introduced in~\cite{MR3534472} 
considered for equations of type~\eqref{SPDE-NemType} that are written in 
the form~\eqref{SPDE}. In this case, the two commutativity 
conditions~\eqref{eq:comm_con_1} and~\eqref{eq:comm_con_2} are fulfilled and the exponential Wagner-Platen type scheme 
in~\cite{MR3534472} reads as $\YWP_0 = X_0$ and
{\allowdisplaybreaks
\begin{align} \label{Com-Exp-Wagner-Platen-scheme}
	\YWP_{m+1} = &\e^{A \tfrac{h}{2}} \bigg\{ \e^{A\tfrac{h}{2}} \YWP_m + h F(\YWP_m) 
		+ \frac{h^2}{2} F'(\YWP_m) [A \YWP_m+F(\YWP_m)] \nonumber \\
		&+ F'(\YWP_m) \bigg( B(\YWP_m) \int_{mh}^{(m+1)h} ( W_s - W_{mh} ) 
		\, \mathrm{d}s \bigg) \nonumber \\
		&+ \sum_{j \in \mathcal{J}} \frac{h^2}{4} F''(\YWP_m)( B(\YWP_m) g_j, B(\YWP_m) g_j ) 
		+ B(\YWP_m) \Delta W_m \nonumber \\
		&+ A \bigg[ B(\YWP_m) \bigg( \int_{mh}^{(m+1)h} ( W_s-W_{mh} ) \, \mathrm{d}s 
		- \frac{h}{2} \Delta W_m \bigg) \bigg] \nonumber \\
		&+ B'(\YWP_m) \Big( A \YWP_m + F(\YWP_m) \Big) \bigg( h \Delta W_m 
		- \int_{mh}^{(m+1)h} ( W_s - W_{mh} ) \, \mathrm{d}s \bigg) \nonumber \\
		&+ \frac{1}{2} B'(\YWP_m) \Big( B(\YWP_m) \Delta W_m \Big) \Delta W_m \nonumber \\
		&+ \frac{1}{6} B''(\YWP_m) \Big( B(\YWP_m)\Delta W_m , B(\YWP_m) \Delta W_m \Big) 
		\Delta W_m \nonumber \\
		&+ \frac{1}{6} B'(\YWP_m) \Big( B'(\YWP_m) \Big( B(\YWP_m) \Delta W_m \Big) 
		\Delta W_m \Big) \Delta W_m \nonumber \\
		&- \frac{h}{2} \sum_{j \in \mathcal{J}} B'(\YWP_m) \Big( B(\YWP_m) g_j \Big) g_j 
		\nonumber \\
		&- \frac{1}{2} \sum_{j \in \mathcal{J}} B''(\YWP_m) 
		\Big( B(\YWP_m)g_j, B(\YWP_m)g_j \Big) \int_{mh}^{(m+1)h} ( W_s - W_{mh} ) 
		\, \mathrm{d}s \nonumber \\
		&- \frac{h}{2} \sum_{j \in \mathcal{J}} B'(\YWP_m) 
		\Big( B'(\YWP_m) \Big( B(\YWP_m) g_j \Big) g_j \Big) \Delta W_m \bigg\}
\end{align}
}%
for $m \in \{ 0, 1, \ldots, M \}$. For details of the used notation, we refer to 
Section~\ref{Sec:Alg}.
Scheme~\eqref{Com-Exp-Wagner-Platen-scheme} converges temporally with 
strong order up to $\sfrac{3}{2}-\varepsilon$ for 
some $\varepsilon>0$ and needs the calculation of first and second order 
derivatives for $F$ and $B$. 
The idea is to construct an exponential stochastic 
Runge-Kutta type method that incorporates the exponential operator from the 
semigroup that generates the mild solution. 
The presented approach is designed for stochastic 
equations in infinite dimensional spaces and it is not directly 
related to the standard class of exponential Runge-Kutta methods for
(stochastic) ordinary differential equations in a finite dimensional setting.
The operator valued first and second order derivatives for $F$ and $B$ are approximated 
by making use of some customized stages motivated by the general concept of 
Runge-Kutta type methods in order to simplify the implementation and to save 
computational cost.
For the derivation of the exponential stochastic 
Runge-Kutta type method, we focus on the notation in~\eqref{SPDE-NemType}. 
As a result of this, we do not need to compute and evaluate any derivatives 
but require only evaluations of the given mappings $f$ and $b$. At the same 
time, we do not lose the high order of convergence, that is, the 
developed exponential stochastic Runge-Kutta type method allows for 
temporal convergence with strong order up to $\sfrac{3}{2}-\varepsilon$
as well.

The benefits of the proposed exponential stochastic Runge-Kutta method 
become even more obvious for systems of stochastic evolution equations
that allow for additional computational savings compared to the exponential
Wagner-Platen type scheme.
It turns out that the computational cost actually depend only linearly on 
the dimension of the considered stochastic evolution equation 
whereas the computational cost of the Wagner-Platen type scheme depend cubically
on that dimension. Thus, the exponential stochastic Runge-Kutta type method
allows for much more efficient numerical approximations compared to the 
Wagner-Platen type scheme and also compared to many other well known schemes as will be shown in the following.
For a further discussion of the benefits of derivative-free methods for stochastic evolution equations, 
we also refer to~\cite{HalRoe2023pre,MR3842926}.

The paper is organized as follows: In Section~\ref{Sec:Setting}, we outline our framework. 
Moreover, we introduce in Section~\ref{Sec:DFM} the structure of the exponential stochastic
Runge-Kutta type method and specify a family of exponential stochastic
Runge-Kutta type schemes with minimized number of stages and computational effort.
This method is explicitly tailored to equations with operators that are 
pointwise multiplicative of Nemytskii-type. 
As the main result, a theorem on the temporal convergence of the exponential stochastic Runge-Kutta type schemes
in the root mean-square sense is stated. The computational cost of the
proposed schemes is compared to the Wagner-Platen type scheme in Section~\ref{Sec:CompCost}.
In Section~\ref{Sec:Numerics} implementation issues based on a spectral Galerkin 
discretization are discussed. Further, the theoretical temporal order of convergence 
for some concrete schemes is illustrated by numerical examples 
in Section~\ref{Sec:ExactSolution}--\ref{Sec:Example3}.
We close this work with detailed proofs on the convergence results of the 
proposed exponential stochastic Runge-Kutta type method in Section~\ref{Sec:proofs}.
\section{The exponential stochastic Runge-Kutta type method}
\label{Sec:Alg}
The numerical method that we present in this section is specifically tailored to
equations which involve operators $F$ and $B$ that are of Nemytskii-type. We 
detail the framework in the next subsection and then introduce the structure of
an exponential stochastic Runge-Kutta type method.
\subsection{Setting}\label{Sec:Setting}
In the following, let $\DomX \subseteq \mathbb{R}^n$ for $n\in\mathbb{N}$ 
denote some 
domain and for $d \in \mathbb{N}$
let $(H,\langle \cdot, \cdot \rangle_H)$ and $(U,\langle \cdot,\cdot \rangle_U)$ 
denote some separable real Hilbert spaces where
$H=L^2(\DomX,\mathbb{R}^d)$ and $U=L^2(\DomX,\mathbb{R})$.
Further, for some $0<T<\infty$, let 
$(W_t)_{t \in [0,T]}$ be a $U$-valued $Q$-Wiener process on a complete 
probability space $(\Omega, \mathcal{F}, \Prob)$ adapted to a filtration
$(\mathcal{F}_t)_{t \in [0,T]}$ that fulfills the usual conditions. The 
corresponding covariance operator $Q \in L(U)$ for the $Q$-Wiener process 
is assumed to be a non-negative and symmetric trace class
operator. Then, $(U_0, \langle \cdot,\cdot \rangle_{U_0})$ with 
$U_0 = Q^{\frac{1}{2}}U$ and 
$\langle u,v \rangle_{U_0} = \langle Q^{-\sfrac{1}{2}} u, Q^{-\sfrac{1}{2}} v \rangle_U$ is a 
real Hilbert space and we denote by $(g_j)_{j \in \mathcal{J}} \subset U_0$
some arbitrary orthonormal basis of $U_0 \subseteq U$ for some countable index 
set $\mathcal{J}$. In order to deal with an arbitrary
orthonormal basis of $U_0$, we assume that $U_0=U$ for simplicity in the following.

We consider the stochastic evolution equation~\eqref{SPDE-NemType} that can
be rewritten as the more abstract equation~\eqref{SPDE}.
The linear operator $A \colon \mathcal{D}(A) \subseteq H \to H$ is assumed to 
have a spectrum $\sigma(A) \subseteq 
\{ \lambda \in \mathbb{C} : \Re(\lambda) < \eta \}$ for some $\eta \geq 0$ and 
to be the generator of an analytic $C_0$-semigroup.

Further, we define a family of interpolation Hilbert spaces 
$H_r := \mathcal{D}((\eta-A)^r)$ for $r \in \mathbb{R}$ with norm 
$\|x\|_{H_r} =\|(\eta-A)^r x\|_H$ for $x \in H_r$ that are associated
to $\eta-A$.
We assume that the initial value
$\xi \colon \Omega \to H_{\gamma}$ for some $\gamma \in [1, \tfrac{3}{2})$ is
$\mathcal{F}_0$-$\mathcal{B}(H_{\gamma})$-measurable.

Let $f \in C^{0,2}(\DomX \times \mathbb{R}^d, \mathbb{R}^d)$ and let the 
operator $F \colon H \to H$ be defined by composition such that
\begin{align*}
	F( v)(x) = f(x, v(x)), \quad v\in H, \ x \in \DomX ,
\end{align*}
and such that $F \in C^2(H,H)$ is globally Lipschitz continuous 
with $F(H_{\alpha}) \subseteq H_{\alpha}$ for some 
$\alpha \in (\gamma-1, \gamma]$.

For the diffusion, let $b \in C^{0,2}(\DomX \times \mathbb{R}^d, \mathbb{R}^d)$ 
and define the operator $B \colon H \to L_{HS}(U_0,H)$ by
\begin{align*}
	(B( v)u)(x) = b(x,v(x)) \cdot u(x), 
	\quad v \in H, \ u \in U_0, \ x \in \DomX ,
\end{align*}
such that $B \in C^2(H, L_{HS}(U_0,H))$ is a globally Lipschitz continuous 
mapping.

Under these assumptions there exists an up to modifications unique mild solution 
\begin{equation*}
	X_t(x) = \e^{At} \xi(x) + \int_0^t \e^{A(t-s)} f(x,X_s(x)) \, \mathrm{d}s
	+ \int_0^t \e^{A(t-s)} b(x,X_s(x)) \, \mathrm{d}W_s(x) 
	\quad \Prob\text{-a.s.}
\end{equation*}
to the stochastic evolution equation~\eqref{SPDE-NemType} for $x \in \DomX$ 
and $t \in [0,T]$, which is
$(\mathcal{F}_t)_{t \in [0,T]}$-predictable, has a continuous modification 
and fulfills $\int_0^T \|X_s\|_H^2 \, \mathrm{d}s < \infty$ $\Prob$-a.s., 
see, e.g.,~\cite[Thm~7.2]{MR3236753}.
\subsection{A family of exponential stochastic Runge-Kutta type schemes and a convergence result}
\label{Sec:DFM}
We propose an exponential stochastic Runge-Kutta type  method 
for the strong approximation
of solutions for stochastic evolution equation~\eqref{SPDE-NemType} with
pointwise multiplicative operators of Nemytskii-type. We abbreviate it as $\ESRK$ method, implying that it is tailored to multiplicative operators. The newly developed $\ESRK$ method is based 
on the Wagner-Platen type scheme~\eqref{Com-Exp-Wagner-Platen-scheme} 
which is a Taylor type scheme incorporating derivatives of the involved
mappings and operators. The main idea for the proposed $\ESRK$ method is 
to replace all derivatives by making use
of evaluations of the corresponding mappings and operators at specially
tailored stages. The $\ESRK$ method is designed in such a way that it 
can attain the same strong temporal order of convergence of up to 
$\sfrac{3}{2}-\varepsilon$ for some $\varepsilon>0$ as the Wagner-Platen 
type scheme. For sake of simplicity, we restrict our considerations 
to equidistant time discretizations.

Let $M \in \mathbb{N}$ be the number of time steps 
defining a grid $0=t_0 < t_1 < \ldots < t_M =T$ on the time interval
$[0,T]$ with step size $h=\frac{T}{M}$ and $t_m=m \, h$ for 
$m\in\{0,1, \ldots, M\}$. Further, for $m \in \{0,1, \ldots, M-1\}$
let $\Delta W_m = W_{(m+1)h} - W_{mh}$ denote the increment of the 
$Q$-Wiener process and let $(Y_m)_{0\leq m \leq M}$ denote the 
approximation process where $Y_m$ is the approximation at time $t_m$.
Then, the $s$-stages $\ESRK$ method for the approximation of solutions
for stochastic evolution equation~\eqref{SPDE-NemType} is defined by 
$\YDFN_0 = X_0$ and
\begin{equation} \label{RKTN-scheme}
	\begin{split}
	\YDFN_{m+1} = \e^{A h} \YDFN_m 
	&+ \e^{A \tfrac{h}{2}} \sum_{i=1}^s \sum_{k=1}^3 \alpha_i^{(k)} \, 
	 f(\cdot, K_i^0) \, \theta_{k}^0
	+ \e^{A \tfrac{h}{2}} \sum_{i=1}^s \sum_{k=1}^5 \beta_i^{(k)} \, 
	b(\cdot, K_i^1) \, \theta_k^1 \\
	&+ \e^{A \tfrac{h}{2}} \, A \sum_{i=1}^s \gamma_i^{(1)} \, 
	b(\cdot, K_i^1) \, \theta_1^2
	\end{split}
\end{equation}
for $m\in\{0,1, \ldots, M-1\}$ with stages 
\begin{equation} \label{RKTN-scheme-stages}
	\begin{split}
	K_i^0 &= \YDFN_m 
	+ \sum_{j=1}^s a_{i,j}^{(0,1)} \, h \, \big( A \, K_j^0 + f(\cdot, K_j^0) \big)
	+ \sum_{j=1}^s \big( b_{i,j}^{(0,1)} \, h + b_{i,j}^{(0,2)} \, \sqrt{h} \big) \, 
	b(\cdot, K_j^1)  , \\
	K_i^1 &= \YDFN_m 
	+ \sum_{j=1}^s a_{i,j}^{(1,1)} \, h \, \big( A \, K_j^0 + f(\cdot, K_j^0) \big)
	+ \sum_{j=1}^s \big( b_{i,j}^{(1,1)} \, h + b_{i,j}^{(1,2)} \, \sqrt{h} \big) \, 
	b(\cdot, K_j^1)
	\end{split}
\end{equation}
for $1 \leq i \leq s$ and random weights
\begin{align*}
	\theta_1^0 &= h, &\quad \theta_2^0 &= \frac{\int_{mh}^{(m+1)h} 
	( W_s - W_{mh} ) \dd s}{h} , &\quad \theta_3^0 &= h \sumJ g_j^2 , 
	\\
	\theta_1^1 &= \Delta W_m , &\quad \theta_2^1 &= \frac{\int_{mh}^{(m+1)h} 
	( W_s - W_{mh} ) \dd s}{h} , &\quad \theta_3^1 &= 
	\sumJ g_j^2 - \frac{(\Delta W_m)^2}{h} , 
\end{align*}
\begin{align*}
	\theta_4^1 &= \frac{1}{h} \bigg( \int_{mh}^{(m+1)h} ( W_s - W_{mh} ) \dd s \sumJ g_j^2 - \frac{1}{3} ( \Delta W_m )^3 \bigg) , &\quad 
	\theta_5^1 &= \Delta W_m \sumJ g_j^2 - \frac{( \Delta W_m )^3}{3h} , \\
	\theta_1^2 &= \int_{mh}^{(m+1)h} ( W_s - W_{mh} ) \dd s 
	- \frac{h}{2} \Delta W_m \, . &\quad &
\end{align*}
The coefficients of the $\ESRK$ method~\eqref{RKTN-scheme}--\eqref{RKTN-scheme-stages} 
can be represented by the following Butcher tableau: 
\renewcommand{\arraystretch}{1.6}
\begin{equation} \label{Butcher-tableau-Ito-1.0}
	{
		\begin{tabular}{|c|c|c}
			$A^{(0,1)}$ & $B^{(0,1)}$ & $B^{(0,2)}$ \\
			\cline{1-3}
			$A^{(1,1)}$ & $B^{(1,1)}$ & $B^{(1,2)}$ \\
			\hline
			\hline
			${\alpha^{(1)}}^T$ & ${\beta^{(1)}}^T$ & ${\beta^{(4)}}^T$ \\
			\cline{1-3}
			${\alpha^{(2)}}^T$ & ${\beta^{(2)}}^T$ & ${\beta^{(5)}}^T$ \\
			\cline{1-3}
			${\alpha^{(3)}}^T$ & ${\beta^{(3)}}^T$ & ${\gamma^{(1)}}^T$
		\end{tabular}
	}
\end{equation}
For the Butcher tableau~\eqref{Butcher-tableau-Ito-1.0},
the weight coefficients are composed to vectors of length $s$ with $\alpha^{(k)} 
= (\alpha_i^{(k)})_{1 \leq i \leq s}$ for $k \in \{1,2,3\}$, $\beta^{(l)} 
= ( \beta_i^{(l)} )_{1 \leq i \leq s}$ for $l \in \{1,2,3,4,5 \}$ and $\gamma^{(1)} 
= ( \gamma_i^{(1)} )_{1 \leq i \leq s}$. The coefficients of the stages are
arranged in $s \times s$-matrices with
$A^{(k,1)} = ( a_{i,j}^{(k,1)} )_{1 \leq i,j \leq s}$ and $B^{(k,l)} = 
( b_{i,j}^{(k,l)} )_{1 \leq i,j \leq s}$ for $k \in \{0,1\}$ and $l \in \{1,2\}$.

The $\ESRK$ method~\eqref{RKTN-scheme}--\eqref{RKTN-scheme-stages} makes use
of random weights $\theta_k^l$ for the main recursion formula~\eqref{RKTN-scheme}
whereas the stages $K_i^0$ and $K_i^1$ use purely deterministic weights only.
The specific choice of the weights allows for some degrees of freedom and
the presented random weights $\theta_k^l$ are chosen such that the number $s$
of necessary stages is minimized. We note that all necessary random variables 
$\theta_k^l$ can be easily simulated by sampling from Gaussian distributions.
Further, we point out that the $\ESRK$ method allows for an arbitrary choice 
of the orthonormal basis $(g_j)_{j \in \mathcal{J}}$ of $U_0$ which is a
valuable feature in practice.

In the following, we restrict our considerations to explicit schemes where the
coefficient matrices $A^{(k,1)}$ and $B^{(k,l)}$ are assumed to be left lower 
triangle matrices. Additionally,
we restrict our analysis to the important case where the computational cost is
minimized by choosing many of the coefficients equal zero such that the number
of necessary evaluations of the functions $f$ and $b$ is as small as possible. 
Here, we note that this does not necessarily lead to a unique class of $\ESRK$ schemes. 
In the following, we want to restrict our considerations to the family of 
explicit $\ESRK$ schemes~\eqref{RKTN-scheme}--\eqref{RKTN-scheme-stages} 
with coefficients given by the Butcher tableau in Table~\ref{SRK-scheme-Alg-Coeff}, which we denote by $\RKTN$. 
Although this family of $\ESRK$ schemes needs $s=6$ stages, most of the coefficients
are zero which additionally reduces computational cost.
The whole family of $\RKTN$ schemes is defined by only seven parameters 
$c_1, \ldots, c_7 \in \mathbb{R} \setminus \{ 0 \}$.
For any choice of these parameters it can be proved that the resulting 
$\RKTN$ scheme converges in the strong sense with temporal order $\gamma \in 
[1,1.5)$ depending on the specific stochastic evolution equation,
see also Section~\ref{Sec:Setting}.
\begin{table}[tbp] 
    \footnotesize
    \renewcommand{\arraystretch}{1.5}
    \[
\hspace*{-0.7cm}    \begin{array}{|cccccc|cccccc|cccccc}
&&&&&&&&&&&&&&&&& \\
        c_1 &&&&& &0&&&&&  &0&&&&& \\
        0&0&&&&  & c_2 &0&&&&  &0&0&&&& \\
        0&0&0&&& &0&0&0&&& &c_3&0&0&&& \\
        0&0&0&0&& &0&0&0&0&& &-c_3&0&0&0&& \\
        0&0&0&0&0& &0&0&0&0&0& &0&0&0&0&0& \\
        \cline{1-18}
        &&&&&  &&&&&& &&&&&& \\
        c_4 &&&&& &0&&&&&  &0&&&&& \\
        0&0&&&&  & c_5 &0&&&&  &0&0&&&& \\
        0&0&0&&& &0&0&0&&&  & -c_6 &0&0&&& \\
        0&0&0&0&& &0&0&0&0&&
        &c_6&0&0&0&& \\
        0&0&0&0&0& &0&0&0&0&0&
        &-\frac{c_7}{c_6}&0&0&0
        & \frac{c_7}{c_6} & \\
        \hline
        \hline
        1-\frac{1}{2 c_1} & \frac{1}{2 c_1} & 0 & 0 & 0 & 0
        & 1-\frac{1}{c_4} & \frac{1}{c_4} & 0 & 0 & 0 & 0
        & \frac{1}{c_6^2} & 0 & 0 & -\frac{1}{2 c_6^2}
        & -\frac{1}{2 c_6^2} & 0 \\
        \cline{1-18}
        \frac{-1}{c_2} & 0 & \frac{1}{c_2} & 0 & 0 &0
        & \frac{1}{c_4} & -\frac{1}{c_4} & 0 & 0 & 0 & 0
        & \frac{1}{2 c_7} & 0 & 0 & 0 & 0 & -\frac{1}{2 c_7} \\
        \cline{1-18}
        \frac{-1}{2 c_3^2} & 0 & 0 & \frac{1}{4 c_3} &
        \frac{1}{4 c_3} & 0
        & \frac{1}{2 c_5} & 0 & \frac{-1}{2 c_5} &  0 & 0 & 0
        & 1 & 0 & 0 & 0 & 0 & 0
    \end{array}
    \]
    \caption{Butcher tableau for $\RKTN$ schemes
    	of order $\gamma \in [1,\tfrac{3}{2})$ with $s=6$ stages and
    	coefficients
		$c_1, c_2, c_3, c_4, c_5, c_6, c_7 \in \mathbb{R} \setminus \{0\}$.}
		\label{SRK-scheme-Alg-Coeff}
\end{table}

Before we state the main result on the strong temporal order of convergence for
the family of $\RKTN$ schemes~\eqref{RKTN-scheme}--\eqref{RKTN-scheme-stages} given by
Table~\ref{SRK-scheme-Alg-Coeff}, an auxiliary result giving a uniform bound
for the $L^p$-moments of the numerical approximation is given.
\begin{prop} \label{Prop:Bounds}
	Let some arbitrary coefficients ${c_1}, {c_2}, {c_3}, {c_4}, {c_5}, {c_6}, {c_7} 
	\in \mathbb{R} \setminus \{0\}$ be given, let $p \geq 2$ and assume that
	\begin{align*}
		K := \, &  \sup_{r\in [0,1]} \etaA{-r} 
		+ \sup_{t\in (0,T]} \ 
		\sup_{\kappa \in \big\{ 0, \tfrac{\theta}{2}, \gamma -\theta, \gamma
			-\tfrac{\theta}{2}, \gamma -\tfrac{\theta}{2}-\delta, 1 \big\} \cap[0,1]} 
		\norm{t^{\kappa} (\eta -A)^{\kappa} \e^{At}}_{L(H)} \\
		&+ \frac{1}{1-\theta} + T + p + \eta +  \sum_{i=1}^7 c_i
		+ \sup_{v\in H_{{\gamma}}} \frac{\|b(\cdot,v)\|_{{H}}}{1+\|v\|_{H_{\gamma}}} \\
		&+ \sup_{v \in H_{\alpha}} \frac{\norm{F(v)}_{H_{\alpha}}}{1+\norm{v}_{H_{\alpha}}} 
		+ \sup_{v\in H} \norm{F'(v)}_{L(H)} 
		+ \sup_{v \in H_{\beta}} \frac{\norm{B(v)}_{HS(U_0,H_{\beta})}}{1 +\norm{v}_{H_{\beta}}} \\
		&
		+ \sup_{v\in H} \norm{B'(v)}_{L(H,HS(U_0,H))}
		+ \sup_{v \in H_{\gamma}} \, 
		\sup_{ w \in H_{\delta} \setminus \{ 0\} }
		\frac{\norm{B'(v)w}_{HS(U_0,H_{\delta})}}{\norm{w}_{H_{\delta}}} \\
		&+\sup_{v,w\in H} \frac{\big(\norm{F''(v)}_{L^{(2)}(H,H)} 
		+   \norm{B''(v)}_{L^{(2)}(H,HS(U_0,H))} \big)
		\norm{B(w)}_{HS(U_0,H)}^2 }{ 1 +\norm{v}_H +\norm{w}_H} < \infty
	\end{align*}
	for some constants $\gamma\in[1,\frac{3}{2})$, $\alpha \in(\gamma-1,\gamma]$, 
	$\beta \in (\gamma -\tfrac{1}{2}, \gamma]$, $\delta \in
	(\gamma-1, \beta]$, $\theta = \max \{ \gamma-\alpha,
	\gamma-\tfrac{1}{2}, 2(\gamma-\beta), 2(\gamma-\delta-\tfrac{1}{2}) \}$ 
	and $\eta \in [0, \infty)$
	such that the assumptions in Section~\ref{Sec:Setting} are fulfilled.
	Then, there exists some non-decreasing function $C \colon [0,\infty) \to [1,\infty)$ 
	such that for the approximation defined by the exponential stochastic 
	Runge-Kutta type schemes~\eqref{RKTN-scheme}--\eqref{RKTN-scheme-stages} given by 
	Table~\ref{SRK-scheme-Alg-Coeff} it holds that
	\begin{equation*}
		\sup_{M \in \mathbb{N}} \ \sup_{m \in \{0,1,\ldots,M\}}
		\normLpH{p}{\gamma}{Y_m} 
		\leq C(K) \, (1+ \| \xi \|_{L^p(\Omega; H_{\gamma})} )  \, .
	\end{equation*}
\end{prop}
The proof of this proposition is given in Section~\ref{Sec:Bound}.
%
%
Next, we analyze the strong temporal convergence of the proposed exponential stochastic Runge-Kutta
type schemes~\eqref{RKTN-scheme}--\eqref{RKTN-scheme-stages} given by Table~\ref{SRK-scheme-Alg-Coeff}. 
The uniform boundedness of the moments 
is essential for the proof of the convergence result given in the next theorem. 
Note that the exponential stochastic Runge-Kutta type schemes~\eqref{RKTN-scheme}--\eqref{RKTN-scheme-stages} 
given by Table~\ref{SRK-scheme-Alg-Coeff} attain a temporal order of convergence $\gamma$ which can take values up to
$\sfrac{3}{2}-\varepsilon$ for some $\varepsilon>0$ depending on the specific SPDE 
that is considered. More precisely, the maximum value of $\gamma$ is 
determined by the ranges for the parameters $\alpha$, $\beta$, $\delta$ and $\theta$ as well as
by the initial value $\xi \colon \Omega\rightarrow H_{\gamma}$ that depend on the considered SPDE.
\begin{thm} \label{Prop:Conv}
	Let some arbitrary coefficients ${c_1}, {c_2}, {c_3}, {c_4}, {c_5}, {c_6}, {c_7} 
	\in \mathbb{R} \setminus \{0\}$ be given
	and assume that
	\begin{align*}
		K := \, & \sup_{M\in\mathbb{N}} \ \sup_{m \in \{0,1,\ldots,M\}}
		\norm{Y_m}_{L^6(\Omega, H_{\gamma}) }
	 	+ \sup_{r\in [0,1]} \etaA{-r} 
	 	+ \sup_{\substack{v,w \in H_{\beta} \\ v \neq w}}
	 	\frac{\| b(\cdot,v)-b(\cdot,w) \|_H}{ \|v-w\|_H } \\
	 	&+ (\min\{\alpha+1,\beta+1/2,\delta+1,3/2\}-\gamma)^{-1} + T + \eta 
	 	+ \ind_{(1,\infty)}(\alpha) \cdot \Big( \frac{1}{\alpha -1} \Big) 
	 	+\sum_{i=1}^7 c_i \\
	 	&+ \ind_{(1,\infty)}(\gamma) \cdot \Big( \frac{1}{\gamma -1} \Big)
	 	+ \sup_{v \in H_{\alpha}} \frac{\norm{F(v)}_{H_{\alpha}}}{1+\norm{v}_{H_{\alpha}}} 
	 	+\sup_{v\in H_{{\gamma}}} \frac{\|b(\cdot,v)\|_{{H}}}{1+\|v\|_{H_{\gamma}}}
 		+ \sup_{v \in H_{\beta}}
 		\frac{\norm{B(v)}_{HS(U_0,H_{\beta})}}{1+\norm{v}_{H_{\beta}}} \\
	 	&+ \sup_{v\in H} \norm{B'(v)}_{L(H,HS(U_0,H))}
	 	+ \sup_{v \in H_{\gamma}} \ \sup_{w \in H_{\delta}\setminus\{0\}}
	 	\frac{\norm{B'(v)w}_{HS(U_0,H_{\delta})}}{\norm{w}_{H_{\delta}}} \\
	 	&+ \sum_{i=0}^2 \sup_{\substack{v,w \in H \\ v\neq w}} \bigg[
	 	\frac{\norm{F^{(i)}(v)-F^{(i)}(w)}_{L^{(i)}(H,H)} 
	 	+ \norm{B^{(i)}(v)-B^{(i)}(w)}_{L^{(i)}(H,HS(U_0,H))}}{\norm{v-w}_H} \bigg] \\
		&+ \sup_{t\in (0,T]} \ \sup_{\kappa \in
		\big\{ 0, \tfrac{1}{2}, 1-\alpha, 2-\alpha, 1-\beta, \tfrac{3}{2}-\beta,
		\tfrac{1}{2}-\delta, 1-\delta, 2-\gamma \big\} \cap[0,1]} 
		\norm{t^{\kappa} (\eta -A)^{\kappa} \e^{At}}_{L(H)} < \infty
	\end{align*}
	for some constants $\gamma\in[1,\frac{3}{2})$, $\alpha \in(\gamma-1,\gamma]$,
	$\beta \in (\gamma -\tfrac{1}{2}, \gamma]$, $\delta \in
	(\gamma-1, \beta]$, $\theta = \max \{ \gamma-\alpha,
	\gamma-\tfrac{1}{2}, 2(\gamma-\beta), 2(\gamma-\delta-\tfrac{1}{2}) \}$ 
	and $\eta \in [0, \infty)$
	such that the assumptions in Section~\ref{Sec:Setting} are fulfilled.
	Then, the approximations defined by the exponential stochastic Runge-Kutta type 
	schemes~\eqref{RKTN-scheme}--\eqref{RKTN-scheme-stages} given by Table~\ref{SRK-scheme-Alg-Coeff}
	converge to the exact solution $(X_t)_{t \in [0,T]}$ of the stochastic evolution
	equation~\eqref{SPDE-NemType} with
	\begin{equation*}
    	\max_{m \in \{0,1, \ldots, M\}}
	    \| X_{t_m} - \YDFN_m \|_{L^2(\Omega,H)} \leq C(K) \, h^{\gamma}
	\end{equation*}
	for any $h= \frac{T}{M},$ $M \in \mathbb{N}$ and some non-decreasing function 
	$C \colon [0,\infty) \to [1,\infty)$ that is independent of $h$.
\end{thm}
We detail the proof for Theorem~\ref{Prop:Conv} in Section~\ref{Sec:Convergence}.
Note that the assumptions in Proposition~\ref{Prop:Bounds} and
Theorem~\ref{Prop:Conv} cover the assumptions from the work on the 
exponential Wagner-Platen type scheme~\eqref{Com-Exp-Wagner-Platen-scheme}
in~\cite[Prop~1, Thm~1]{MR3534472} as we use 
their proof in an intermediate step, see Section~\ref{Sec:proofs} for details. 
Moreover, we impose some additional conditions on $b$ and tighten the 
assumption on the term 
\begin{align*}
	\sup_{v,w\in H} \frac{\big(\norm{F''(v)}_{L^{(2)}(H,H)} 	
		+ \norm{B''(v)}_{L^{(2)}(H,HS(U_0,H))} \big)
	\norm{B(w)}_{HS(U_0,H)}^2 }{ 1 +\norm{v}_H +\norm{w}_H} < \infty
\end{align*}
slightly for our proof of Proposition~\ref{Prop:Bounds}.
\subsection{Computational cost of the exponential stochastic Runge-Kutta type schemes}
\label{Sec:CompCost}
Now, we want to compare the introduced exponential stochastic Runge-Kutta 
type schemes~\eqref{RKTN-scheme}--\eqref{RKTN-scheme-stages} with coefficients 
given in Table~\ref{SRK-scheme-Alg-Coeff} to the exponential Wagner-Platen type
scheme~\eqref{Com-Exp-Wagner-Platen-scheme} which essentially is a stochastic
Taylor type scheme. These schemes are developed for a temporal discretization
of a stochastic evolution equation of type~\eqref{SPDE-NemType} and can be 
combined with some spatial 
discretization for a fully discretized implementation. Since the considered 
numerical schemes can in principle be combined with various spatial discretization 
approaches, we restrict our comparison to the purely time discretized schemes 
as they are given in, e.g., \eqref{Com-Exp-Wagner-Platen-scheme} and 
\eqref{RKTN-scheme}--\eqref{RKTN-scheme-stages}. Therefore, we compare the 
computational effort of these schemes for one arbitrary time step if they
are applied to an $L^2(\DomX,\mathbb{R}^d)$-valued stochastic evolution 
equation~\eqref{SPDE-NemType} for some $d \in \mathbb{N}$ in the following.

Since for the two considered schemes the operators $A$, $e^{Ah}$ and
$e^{A \frac{h}{2}}$ have to be applied two times in
each step, the computational cost for these operations are the same
for both schemes. 
Moreover, the computational cost for the simulation of the involved random
variables based on the $Q$-Wiener process are exactly the same for the 
schemes under consideration and thus make no difference for the comparison.
Finally, the computational effort corresponding to the evaluation of the basis 
functions $g_j$ for $j \in \mathcal{J}$ is also the same and
can be neglected asymptotically as these functions do not change over time and thus have to be evaluated at some $x \in \DomX$ only once.
Thus, for a comparison of the computational effort for the $\RKTN$ schemes
and the exponential Wagner-Platen type scheme, we focus on the number of necessary
evaluations of functionals, i.e., evaluations of real valued functions,
see also \cite{HalRoe2023pre,MR3842926} for a detailed reasoning of a 
similar cost model. For the functions $f,b \colon \DomX \times \mathbb{R}^d 
\to \mathbb{R}^d$ in stochastic evolution equation~\eqref{SPDE-NemType}
with $f=(f^k)_{1 \leq k \leq d}$ and $b=(b^k)_{1 \leq k \leq d}$,
we compare the number of necessary evaluations for each real valued function 
$f^k \colon \DomX \times H \to \mathbb{R}$ and $b^k \colon \DomX \times H \to \mathbb{R}$ 
for $1 \leq k \leq d$ in order to compare the computational cost for
the $\RKTN$ schemes and the exponential Wagner-Platen type scheme.

Considering the $\RKTN$ schemes~\eqref{RKTN-scheme} with stages~\eqref{RKTN-scheme-stages} 
and coefficients in Table~\ref{SRK-scheme-Alg-Coeff}
one can easily see that $5$~evaluations of each $f^k$ and $6$~evaluations of each
$b^k$ for $k\in\{1, \ldots, d\}$ are necessary in each time step if the approximation needs
to be evaluated at any $x \in \DomX$. Thus, the computational cost for any
$\RKTN$ scheme is given by $11 d$~real valued function evaluations for 
each time step.

On the other hand, for the exponential Wagner-Platen type scheme one needs to 
evaluate the real valued functions $f^k$, $\tfrac{\partial}{\partial y^i} f^k$ and
$\tfrac{\partial^2}{\partial y^i \partial y^j} f^k$ as well as $b^k$,
$\tfrac{\partial}{\partial y^i} b^k$ and $\tfrac{\partial^2}{\partial y^i \partial y^j} b^k$
for $i,j,k \in \{1, \ldots, d\}$. Considering~\eqref{Com-Exp-Wagner-Platen-scheme},
the computational cost for the exponential Wagner-Platen type scheme is given 
as 
$2(d+d^2+d^3)$ which is the number of necessary evaluations of real valued 
functions for each evaluation of the approximation at some $x \in \DomX$.

Comparing the computational effort for the up to order $\sfrac{3}{2}$~schemes, 
we can see that for $d=1$ the computational effort for the exponential 
Wagner-Platen type scheme is with $6$ evaluations of real valued functions 
less than that for the $\RKTN$ schemes with $11$ evaluations of real valued 
functions. However, the main advantage of the $\RKTN$ schemes is that the 
first and second derivatives of $f^k$ and $b^k$ do not need to be calculated
which may justify the application of the $\RKTN$ schemes anyway. 
The situation changes significantly for $d \geq 2$ where the computational 
effort for the $\RKTN$ schemes is always less than that for the exponential 
Wagner-Platen type scheme. Noting that the computational effort grows linearly 
with $d$ for the $\RKTN$ schemes, the computational savings can be dramatic 
for higher dimensional systems as the computational effort for the exponential 
Wagner-Platen type scheme grows cubically with the dimension $d$. This is a 
typical situation that arises for Taylor type approximation methods.
\section{Implementation and numerical examples}\label{Sec:Numerics}
For an implementation of the proposed $\RKTN$ schemes~\eqref{RKTN-scheme}--\eqref{RKTN-scheme-stages} 
defined by Table~\ref{SRK-scheme-Alg-Coeff}
and all other
numerical schemes that we consider for a comparison, we implement
a spectral Galerkin discretization of the space as an example. 
Further, we discuss the simulation of all involved random variables
based on the underlying $Q$-Wiener process. In Section~\ref{Sec:NumEx}, 
some numerical examples are considered in order to demonstrate 
the theoretical findings and the performance of the proposed $\RKTN$
schemes compared to some well known numerical schemes. 
\subsection{Spectral Galerkin implementation}
\label{Sec:Implementation}
In order to implement the proposed $\RKTN$ schemes~\eqref{RKTN-scheme}--\eqref{RKTN-scheme-stages}, 
we combine them with a 
spatial spectral Galerkin approximation method on the state space. 
Here, we assume that there exist some countable index set $\mathcal{I}$ and 
a family $(\lambda_i)_{i \in \mathcal{I}}$ of real numbers with 
$\inf_{i \in \mathcal{I}} \lambda_i > -\eta$ for some $\eta \geq 0$ 
of eigenvalues of $-A$ and an orthonormal basis of eigenvectors 
$(e_i)_{i \in \mathcal{I}}$ for $H$ 
such that $-A e_i = \lambda_i e_i$ for $i \in \mathcal{I}$ and
$-A v = \sum_{i \in \mathcal{I}} \lambda_i \langle v, e_i \rangle_H e_i$ for all
$v \in \mathcal{D}(A)$.
Then, we consider projections of the temporal approximations on a finite 
dimensional subspace of $H$ as well as a projection of the 
driving stochastic process $(W_t)_{t\in[0,T]}$ on a finite dimensional 
subspace of $U$. We therefore introduce the 
projection operator $P_N \colon H \to H_N$ for some $N \in \mathbb{N}$ given as
\begin{equation} \label{Eqn:P-N-Projection}
	P_N h = \sum_{i \in \mathcal{I}_N} \langle h, e_i \rangle_H e_i, 
	\quad h \in H,
\end{equation}
for some subset $\mathcal{I}_N \subseteq \mathcal{I}$ with $| \mathcal{I}_N | = N$.
Here, $H_N = \Span \{ e_i : \ i \in \mathcal{I}_N \}$ denotes some $N$-dimensional
subspace of $H$.
Further, we choose $(\tilde{e}_j)_{j \in \mathcal{J}}$ for some countable index set 
$\mathcal{J}$ to be an ONB of $U$ such that 
$Q \tilde{e}_j = \eta_j \tilde{e}_j$ for some $\eta_j \geq 0$ and $j \in \mathcal{J}$.
In the following, we set $g_j = \sqrt{\eta_j} \tilde{e}_j$ for $j \in \mathcal{J}$ as an example
which is possible since we assume that $U_0=U$.
However, note that in general $(g_j)_{j \in \mathcal{J}}$ can be an arbitrarily chosen 
ONB for $U_0$ that can be different to 
$(\sqrt{\eta_j} \tilde{e}_j)_{j \in \mathcal{J}}$.
For some $K \in \mathbb{N}$ and some subset $\mathcal{J}_K \subset \mathcal{J}$ 
such that $| \mathcal{J}_K | = K$ we denote by
\begin{equation} \label{Eqn:W-N-Projection}
	W_t^K  = \sum_{j \in \mathcal{J}_K} \sqrt{\eta_j} \, \beta_t^j \, \tilde{e}_j , \quad t \in [0,T],
\end{equation}
the orthogonal projection of the $Q$-Wiener process $(W_t)_{t \in [0,T]}$ on the $K$-dimensional
subspace $U_K = \Span \{ \tilde{e}_j : \ j \in \mathcal{J}_K \} \subseteq U$. Moreover, let
$\Delta W_m^K = W_{(m+1)h}^K - W_{mh}^K$ denote the increment of the projected $Q$-Wiener process
on the time interval $[t_m,t_{m+1}]$.
Then, we consider the projected $\ESRK$
approximation given by $\YDFN_0^{M,N,K} = P_N X_0$ and
\begin{equation} \label{RKTN-scheme-PN}
	\begin{split}
		\YDFN_{m+1}^{M,N,K} = P_N \, \e^{A \tfrac{h}{2}} \Big( \e^{A \tfrac{h}{2}} \YDFN_m^{M,N,K}
		&+ \sum_{i=1}^s \sum_{k=1}^3 \alpha_i^{(k)} \, f(\cdot, K_i^0) \, \theta_k^0
		+ \sum_{i=1}^s \sum_{k=1}^5 \beta_i^{(k)} \, b(\cdot, K_i^1) \, \theta_k^1 \\
		&+ A \sum_{i=1}^s \gamma_i^{(1)} \, b(\cdot, K_i^1) \, \theta_1^2 \Big)
	\end{split}
\end{equation}
with stages 
\begin{equation} \label{RKTN-scheme-stages-PN}
	\begin{split}
		K_i^0 &= P_N \Big( \YDFN_m^{M,N,K}
		+ \sum_{j=1}^s a_{i,j}^{(0,1)} \, h \, \big( A \, K_j^0 + f(\cdot, K_j^0) \big)
		+ \sum_{j=1}^s \big( b_{i,j}^{(0,1)} \, h + b_{i,j}^{(0,2)} \, \sqrt{h} \big) \, 
		b(\cdot, K_j^1) \Big) ,\\
		K_i^1 &= P_N \Big( \YDFN_m^{M,N,K}
		+ \sum_{j=1}^s a_{i,j}^{(1,1)} \, h \, \big( A \, K_j^0 + f(\cdot, K_j^0) \big)
		+ \sum_{j=1}^s \big( b_{i,j}^{(1,1)} \, h + b_{i,j}^{(1,2)} \, \sqrt{h} \big) \, 
		b(\cdot, K_j^1) \Big)
	\end{split}
\end{equation}
for $1 \leq i \leq s$ and $m\in\{0,1, \ldots, M-1\}$, 
where we replace the $Q$-Wiener process $W$ by the 
projected process $W^K$ and where we replace the index set $\mathcal{J}$ by
$\mathcal{J}_K$ in the expressions for the random weights $\theta_2^0$, 
$\theta_3^0$, $\theta_1^1$, $\theta_2^1$, $\theta_3^1$, $\theta_4^1$,
$\theta_5^1$ and $\theta_1^2$.
For the spectral Galerkin projected version of the exponential Wagner-Platen type
scheme, we refer to \cite[Eqn.~(21)]{MR3534472}, which can be calculated in the
same manner as for the $\RKTN$ schemes.
\subsection{Simulation of mixed integrals}
\label{Sec:MixedIntegrals}
For an implementation of the proposed $\RKTN$ schemes~\eqref{RKTN-scheme}--\eqref{RKTN-scheme-stages}, 
we face the problem that 
it involves mixed integrals of the form
\begin{equation*}
	\int_{mh}^{(m+1)h} (W_s-W_{mh}) \, \mathrm{d}s
\end{equation*}
for $m \in \{0,\ldots,M-1\}$ as they appear in the random variables 
$\theta_2^0$, $\theta_2^1$, $\theta_4^1$, $\theta_1^2$, see Section~\ref{Sec:DFM}. 
In order to compute a step with the numerical method, one has to take 
care of their joint simulation in accordance with the 
increments of the $Q$-Wiener process $\Delta W_m$ which are involved 
in many of the other random variables, e.g., $\theta_1^1$, $\theta_3^1$. 

Note that this is not special for the $\RKTN$ schemes; rather these random variables 
have to be simulated in the exponential Wagner-Platen type 
scheme~\eqref{Com-Exp-Wagner-Platen-scheme} as well. In~\cite{MR3534472}, the 
authors therefore give the following lemma, see \cite[Lemma~3]{MR3534472} for a proof.
\begin{lem}\label{Lem:Mixed}
	Let $Q \colon U \to U$ be a non-negative symmetric trace class operator and 
	let $(W_t)_{t\in[0,T]}$ be a $Q$-Wiener process taking values in $U$. 
	Then it holds for all $t_0,t\in[0,T]$ with $t_0\leq t$ and all $u_1,u_2\in U$ that
	\begin{equation*}
		\left(\Cov\left(\begin{array}{c}
			W_t-W_{t_0} \\ 
			\int_{t_0}^t(W_s-W_{t_0})\,\mathrm{d}s
		\end{array}
		\right)\right)
		\left(\begin{array}{c}
			u_1\\ u_2
		\end{array}
		\right)
		= \left(\begin{array}{c}
			(t-t_0)Qu_1+ \frac{1}{2}(t-t_0)^2 Q u_2 \\
			\frac{1}{2} (t-t_0)^2 Q u_1 + \frac{1}{3}(t-t_0)^3 Q u_2
		\end{array}
		\right).
	\end{equation*}
\end{lem}
We employ this lemma for the simulation of our numerical examples below 
in the following way. It is well known that a vector of real-valued random 
variables ${\xi}\in\mathbb{R}^n$, for some $n\in\mathbb{N}$, which is 
normally distributed with mean $\mu$ and covariance matrix $C$, can be 
obtained from a vector of standard normally distributed random variables 
$z\in\mathbb{R}^n$ by ${\xi} = \mu + C^{1/2} {z}$. Here, $C^{1/2}$ 
denotes the Cholesky decomposition of the matrix $C$. 

In our simulations, we consider the projected process 
$W_t^K = \sum_{j = 1}^K \sqrt{\eta_j} \, \beta_t^j \, \tilde{e}_j$, $t\in[0,T]$, 
that is, a finite sum of real-valued random variables. In order to compute 
the mixed integrals involving this process and its increments simultaneously, 
we compute the Cholesky decomposition of the covariance matrix 
\begin{equation*}
	C = \begin{pmatrix}
		h & \frac{1}{2} h^2 \\
		\frac{1}{2} h^2 & \frac{1}{3}h^3
	\end{pmatrix} ,
\end{equation*} 
according to Lemma~\ref{Lem:Mixed} with $t_0 = t_m$ and $t=t_{m+1}$.
We then draw a matrix ${(z_1,z_2)^\top} \in \mathbb{R}^{2K\times M}$ of standard 
normally distributed random variables. This allows to compute  
\begin{equation*}
	\begin{pmatrix}
		\beta_{t_{m+1}}^j -\beta_{t_m}^j \\ 
		\int_{mh}^{(m+1)h}(\beta^j_s-\beta^j_{t_m}) \, \mathrm{d}s 
	\end{pmatrix}
	= \begin{pmatrix} 
		{h}^{\sfrac{1}{2}} & 0 \\
		\frac{1}{2} h^{\sfrac{3}{2}} & \frac{1}{2 \sqrt{3}} h^{\sfrac{3}{2}}
	\end{pmatrix} 
	\begin{pmatrix}
		{(z_1})_m^j \\ 
		{(z_2)}_m^j
	\end{pmatrix}
\end{equation*}
for all $j\in\{1,\ldots,K\}$ and $m\in\{0,\ldots,M-1\}$.
With these random variables, we compute $\Delta W^K_m$ and $\int_{mh}^{(m+1)h}(W^K_s-W^K_{t_m})\,\mathrm{d}s$ 
by multiplying with the square-root of the eigenvalues $\sqrt{\eta_j}$ and eigenfunctions 
$\tilde{e}_j$ for $j\in\{1,\ldots,K\}$ and then adding them up respectively, for each time step.
\begin{remark}
	In this work, we consider the case where mapping $F$ and the pointwise 
	multiplicative operator $B$ are of
	Nemytskii-type. We do not have to simulate any iterated stochastic integrals
	in this setting as the commutativity conditions needed in~\cite{MR3534472}
	for the exponential Wagner-Platen type scheme or the Milstein type scheme 
	in~\cite{MR3320928,MR3842926} are naturally fulfilled. This
	allows us to express the iterated stochastic integrals in terms of increments of
	the $Q$-Wiener process.
\end{remark}
\subsection{Numerical examples}
\label{Sec:NumEx}
For an analysis of performance, we compare the $\RKTN$ schemes to 
the linear implicit Euler ($\LIE$) scheme~\cite{MR1825100}, the exponential 
Euler ($\EXE$) scheme, see e.g.~\cite{MR3047942}, the derivative-free 
Milstein type ($\SESRK$) scheme~\cite{MR3842926}
and the exponential Wagner-Platen type ($\EWP$) scheme~\cite{MR3534472} 
for different numerical examples.
For simplicity, we choose 
$c_1= c_2 = c_3 = c_4 = c_5 = c_6 = c_7=1$ as the coefficients for the $\RKTN$ 
scheme in our comparison and denote this scheme as $\RKTNV$. 

In the following examples, let $H=L^2(\DomX,\mathbb{R})$ with $\DomX=(0,1)$ 
and let $A \colon \mathcal{D}(A) \subseteq H \to H$ with $A = \kappa \Delta$
be the Laplacian multiplied by some constant $\kappa>0$ with Dirichlet boundary
conditions on $H$. As an ONB $\{e_k : k \in \mathbb{N} \}$ on $H$, we choose
the eigenfunctions $e_k(x) = \sqrt{2} \sin(k \pi x)$ for $x \in \DomX$ of $-A$ and 
denote the corresponding eigenvalues by $\lambda_k = \kappa \pi^2 k^2$ for
$k \in \mathbb{N}$.
For all numerical schemes under consideration, we apply a spatial 
spectral Galerkin discretization as described in 
Section~\ref{Sec:Implementation} with $N=256$ and $\mathcal{I}_N = \{1, \ldots, N\}$.
For the considered $U$-valued $Q$-Wiener process, we always take 
$\mathcal{J}_K = \{1, \ldots, K\}$ for some $K \in \mathbb{N}$ and use $W_t^K$ 
as an approximation for $W_t$, see Section~\ref{Sec:Implementation}.
For each test example, we compare
the $L^2(\Omega,H)$-error of the numerical approximations computed by the 
different schemes under consideration at the final time point $T=1$. 
For the computation of the $L^2(\Omega,H)$-errors, $2000$ independent
realizations of the solution and the numerical approximations are simulated and
then the arithmetic mean is used as an estimator for the expectation of the errors.
Some of the considered SPDEs do not allow for an analytical solution. In case that there 
exists no explicitly given mild solution for the SPDE under consideration, we 
apply the exponential Wagner-Platen type scheme with $M=2^{16}$ equidistant 
time steps in order to obtain a highly accurate approximation that serves as a 
reference solution. In that case, this reference solution is used as a substitution for the unknown
exact solution in the error criterion in order to compare the accuracy and
performance of the numerical schemes.
\subsubsection{Example with an explicitly given solution}
\label{Sec:ExactSolution}
%
%
First, we consider an example with multiplicative noise as 
from~\cite[Sec 5.1]{MR3842926} which allows for an exact solution. 
We assume $(W(t))_{t \in [0,T]}$ to be a scalar Brownian motion 
$(\beta_t)_{t\geq 0}$ by setting $U=\mathbb{R}$., i.e., we consider the case
that $W_t=W_t^K$ for $K=1$ with $\eta_1=1$ and eigenfunction $g_1 \equiv 1$. 
Then, the SPDE that we want to solve reads as
\begin{align*} 
	&\mathrm{d}X_t(x) = \Delta X_t(x) \, \mathrm{d}t + X_t(x) \, \mathrm{d}\beta_t,
	&\quad &t>0, \, x \in \DomX, \notag \\
	&X_0(x) = \sqrt{2} \sum_{n\in\mathbb{N}} n^{-4} \sin(n \pi x), &\quad &x \in \DomX, \\
	&X_t(0) = X_t(1) = 0, &\quad &t\geq 0. \notag
\end{align*}
Here, $\kappa=1$ and the mapping 
$b \colon \DomX \times \mathbb{R} \to \mathbb{R}$ is given by
$b(x,v(x)) = v(x)$ for $v \in H$ and $x \in \DomX$. 
For this SPDE, the solution is given by
\begin{equation*}
	X_t(x) = \sqrt{2} \sum_{n\in\mathbb{N}}
	n^{-4} \, e^{-(n^2\pi^2+\frac{1}{2})t+\beta_t} \sin(n\pi x)
\end{equation*}
for $x \in \DomX$ and $t \geq 0$. One can verify for this example that all 
assumptions from 
Section~\ref{Sec:Alg}, Proposition~\ref{Prop:Bounds} as well as 
Theorem~\ref{Prop:Conv} are fulfilled for any $\gamma \in [1,\frac{3}{2})$.
\begin{figure}[tbp] 
	\begin{center}
		\includegraphics[scale = 0.38]{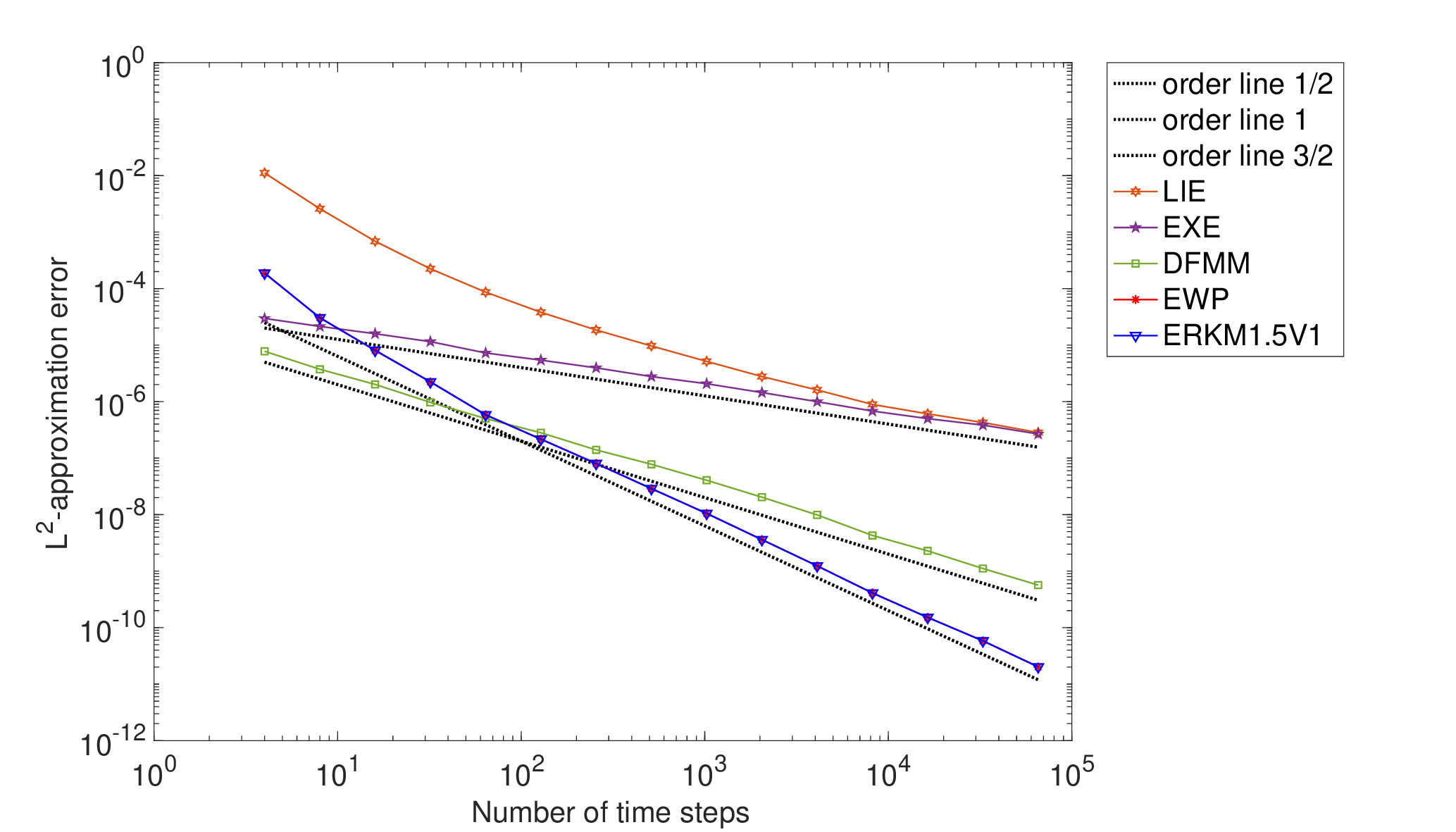}
	\end{center}
	\caption{$L^2(\Omega,H)$-error at time $T=1$ versus number of time steps 
		in $\log$-$\log$-scale
		for $M = 2^l$,  $l\in\{2,3,\cdots,16\}$, for Example~1 in
		Section~\ref{Sec:ExactSolution}.}
	\label{Fig:Ex1}
\end{figure}
For each numerical scheme, we compute numerical 
approximations based on $M = 2^l$ equidistant time steps for 
$l \in \{2,3 \ldots, 16\}$ 
on the time interval $[0,1]$ and compare the corresponding
$L^2(\Omega,H)$-errors.
Since $\gamma \in [1,\frac{3}{2})$ can be arbitrarily chosen, we expect
convergence with order of approximately $\sfrac{3}{2}$ for both, the $\EWP$
scheme and the $\RKTNV$ scheme. Figure~\ref{Fig:Ex1} shows that the 
$\EXE$ and the $\LIE$ scheme have a temporal order
of convergence $\sfrac{1}{2}$, the $\SESRK$ attains order of 
convergence~$1$ and the $\EWP$ scheme as well as the $\RKTNV$ scheme 
achieve the highest temporal order of convergence close to $\sfrac{3}{2}$. 
Both, the $\EWP$ and the $\RKTNV$ scheme yield the same convergence 
results. Thus, as can be seen in Figure~\ref{Fig:Ex1}, we observe that all 
schemes converge with their expected theoretical orders.
\subsubsection{Stochastic heat equation with linear multiplicative noise}
\label{Sec:Example2}
The second example is the stochastic heat equation with multiplicative noise that 
is considered in \cite{MR3534472}. 
Let $U=H$,
let $A=\sfrac{1}{10} \, \Delta$, i.e., $\kappa = \sfrac{1}{10}$, 
and let $Q=(-A)^{-3} \in L(H)$ denote the covariance operator for the
$Q$-Wiener process $(W(t))_{t \in [0,T]}$. 
We consider the following SPDE
\begin{align*}
	&\mathrm{d}X_t(x) = \frac{1}{10} \Delta X_t(x) \,\mathrm{d}t + X_t(x) \, \mathrm{d}W_t(x),
	&\quad &t \in (0,T], \, x \in \DomX, \\
	&X_0(x) = \sin(\pi x), &\quad &x \in \DomX, \\
	&X_t(0) = X_t(1) = 0, &\quad &t \in [0,T].
\end{align*}
Here, the mapping $b \colon \DomX \times \mathbb{R} \to \mathbb{R}$ is 
given by $b(x,v(x)) = v(x)$
for $v \in H$ and $x \in \DomX$.
For this SPDE nearly all assumptions of Section~\ref{Sec:Alg} 
and especially in Proposition~\ref{Prop:Bounds} as well as in
Theorem~\ref{Prop:Conv} can be easily checked to be fulfilled.
Only the assumption $\sup_{v \in H_{\beta}} 
\frac{\| B(v) \|_{HS(U_0,H_{\beta})}}{1+\|v\|_{H_{\beta}}}
< \infty$ is not obviously clear to be fulfilled and we refer to \cite[Ex~1]{MR3534472} 
for a discussion of this condition for $\beta \in (\frac{1}{2},1)$ and 
$\gamma \in [\beta, \beta+\frac{1}{2}) \cap [1,\frac{3}{2})$. 
Note that there exists a unique mild solution to the considered SPDE. 
\begin{figure}[tbp] 
	\begin{center}
		\includegraphics[scale = 0.38]{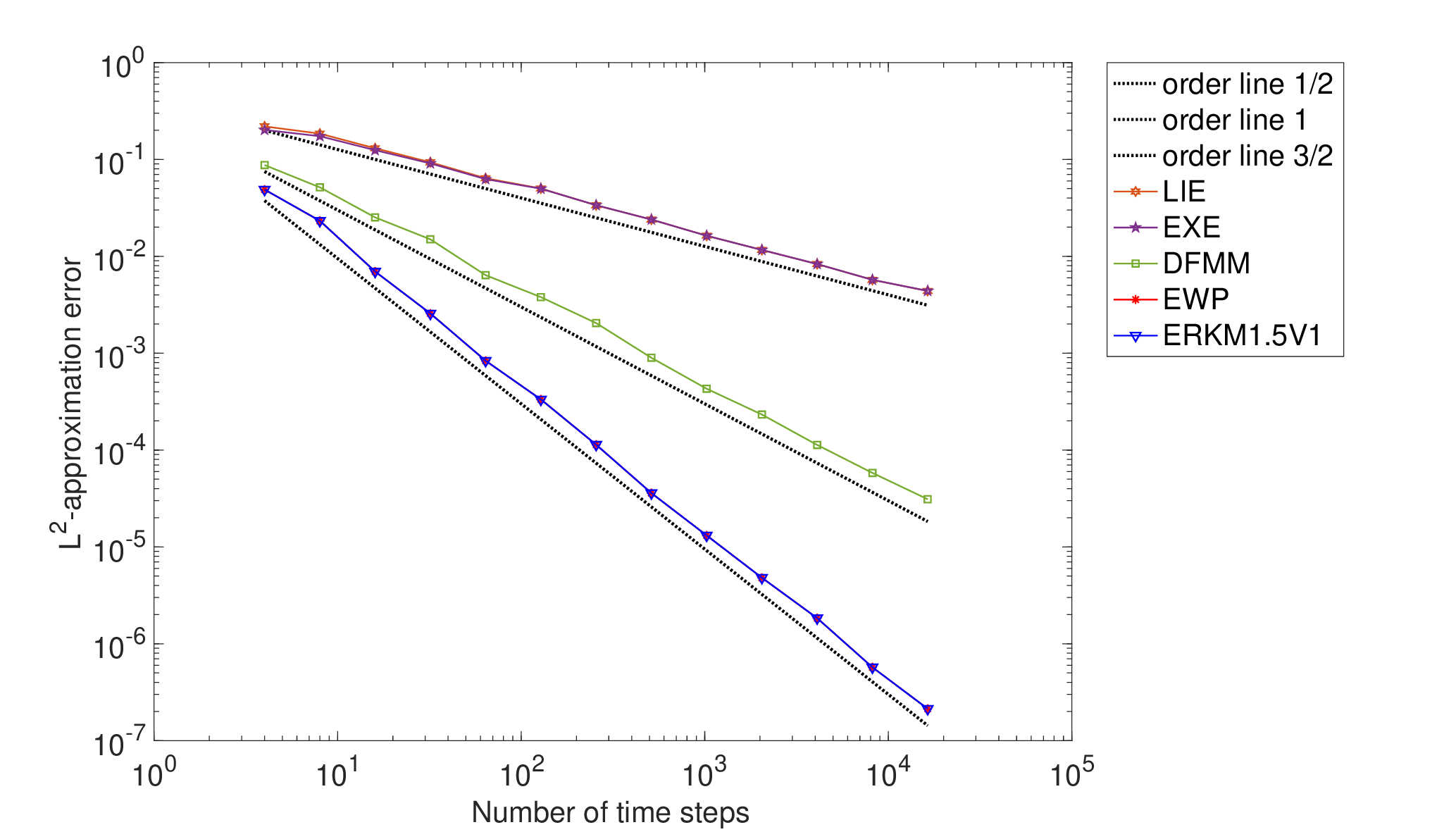}
	\end{center}
	\caption{$L^2(\Omega,H)$-error at $T=1$ versus number of time steps for
		$M = 2^l$, $l\in\{2,3,\cdots,14\}$ for Example~2 in Section~\ref{Sec:Example2}.}
	\label{Fig:Ex2}
\end{figure}

For the numerical simulations, we set $K= 256$ and approximate $W_t$ by 
$W_t^K$. Further, we make use of a reference solution calculated by the exponential
Wagner-Platen type scheme because an explicitly given solution is not known for this 
equation. All considered numerical schemes are applied with step sizes $2^{-l}$
for $l \in \{2,3,\ldots, 14\}$ in order to compare their orders of convergence.
For this example, we expect for the $\EWP$ scheme and the $\RKTNV$ scheme 
a temporal convergence with any order $\gamma \in [1, \frac{3}{2})$, i.e., 
we expect an order up to $\sfrac{3}{2}$. In Figure~\ref{Fig:Ex2}, one observes a 
difference in the temporal order of convergence across the schemes where 
the $\LIE$ and the $\EXE$ scheme attain order $\sfrac{1}{2}$, the $\SESRK$
scheme obtains order $1$ and where the $\EWP$ scheme and the $\RKTNV$ 
scheme achieve an order of $\sfrac{3}{2}$ which confirms the theoretical findings. 
Thus, the $\EWP$ scheme and the $\RKTNV$ scheme yield the same results with
the highest order of convergence compared to the other schemes for this example.
\subsubsection{A nonlinear example}
\label{Sec:Example3}
We choose $U=H$ and we consider an SPDE with operator 
$A = \sfrac{1}{100} \Delta$ and nonlinear coefficients given by
\begin{align*}
	&\mathrm{d}X_t(x) = \Big( \frac{1}{100} \Delta X_t(x) +\sin(X_t(x)) \Big) 
	\, \mathrm{d}t +\cos(X_t(x)) \, \mathrm{d}W_t(x) , &\quad &t \in (0,T], x \in \DomX, \\
	&X_0(x) = \frac{1}{2} \sin(2 \pi x), &\quad &x \in \DomX, \\
	&X_t(0) = X_t(1) = 0, &\quad &t \in [0,T].
\end{align*}
For this equation, it holds $\kappa = \sfrac{1}{100}$ and 
the mappings $f,b \colon \DomX \times \mathbb{R} \to \mathbb{R}$
are defined by $f(x,v(x))=\sin(v(x))$ and $b(x,v(x))=\cos(v(x))$ for $v \in H$ 
and $x \in \DomX$. The covariance operator $Q$ is assumed
to have eigenfunctions $ \{\tilde{e}_j  : j \in \mathbb{N} \}$ with 
$\tilde{e}_j (x)=\sqrt{2} \sin(j \pi x)$
for $x \in \DomX$ and corresponding eigenvalues $\eta_j = j^{-3}$ for $j \in \mathbb{N}$.
Here, we choose $K=256$ for the approximation of the $Q$-Wiener process.
Note that the considered SPDE has a unique mild solution. 
We point out that not all of the assumptions
in Section~\ref{Sec:Alg} and especially in Proposition~\ref{Prop:Bounds} as well as in
Theorem~\ref{Prop:Conv} are fulfilled. Therefore, convergence for the $\EWP$ scheme 
and the $\RKTNV$ scheme is not guaranteed by the theoretical results.
Note that, as for the Wagner-Platen type scheme in~\cite[p.~2399]{MR3534472}, 
a nonlinear multiplication operator $B$ does in general not fulfill the assumptions 
in Section~\ref{Sec:Setting}.

\begin{figure}[tbp] 
	\begin{center}
		\includegraphics[scale = 0.38]{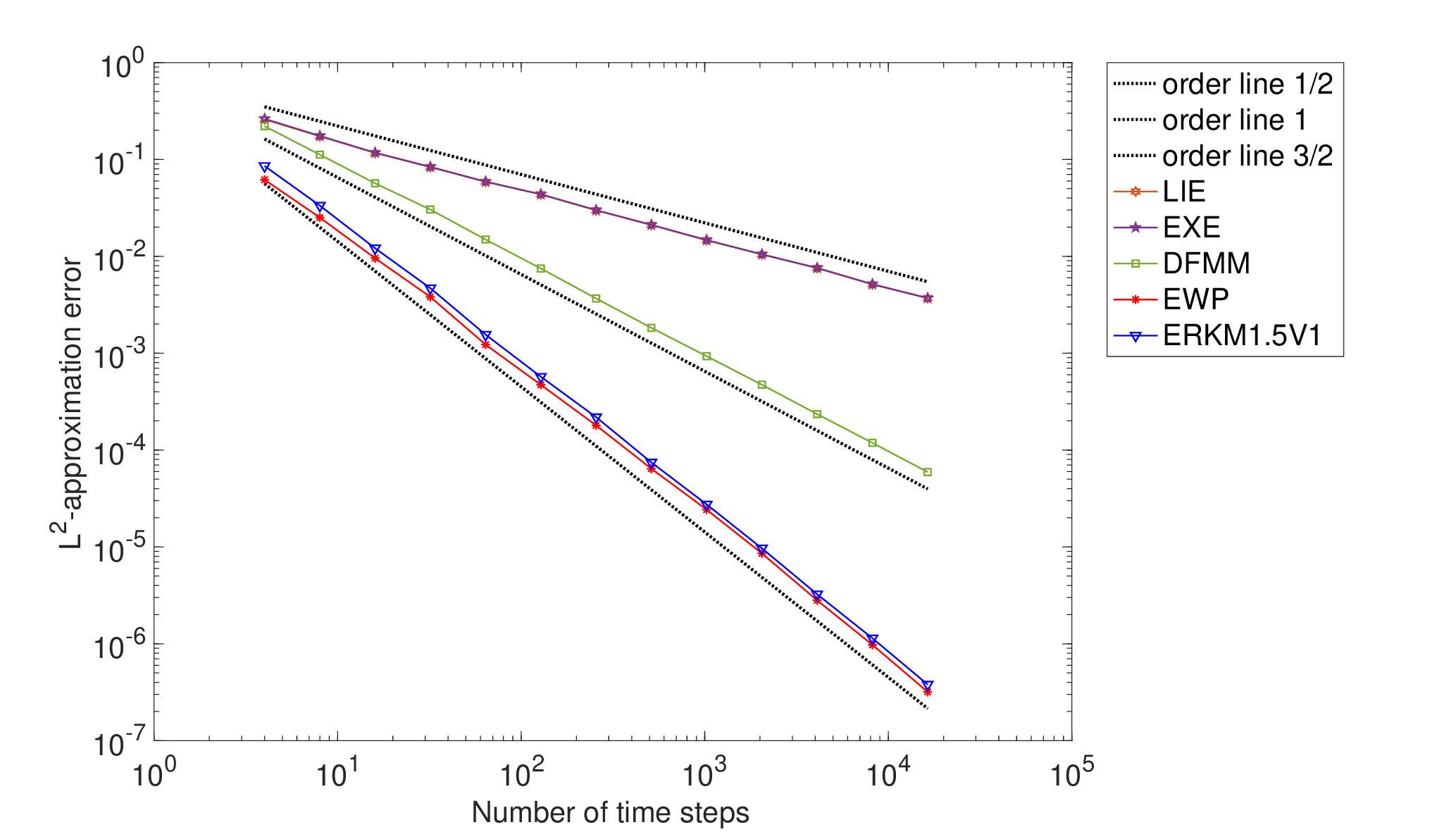}
	\end{center}
	\caption{$L^2(\Omega,H)$-error at $T=1$ versus number of time steps $M = 2^l$, $l\in\{2,3,\cdots,14\}$ 
		for Example 3 in Section~\ref{Sec:Example3}.}
	\label{Fig:Ex3}
\end{figure}

For the equation under consideration there exists no explicitly given solution.
Therefore, the exponential Wagner-Platen type scheme 
is applied for the computation of a reference solution. All considered schemes 
are applied with step sizes $2^{-l}$ for $l \in \{2,3, \ldots, 14\}$ for
comparison of their accuracy. The plot in Figure~\ref{Fig:Ex3} 
indicates convergence for all these schemes which supposes
the theoretical convergence results for the $\EWP$ scheme and the 
$\RKTNV$ scheme not to be sharp concerning the assumptions.
Further, Figure~\ref{Fig:Ex3} nicely confirms the difference in the 
temporal order of convergence for the different schemes.
For this example, the $\LIE$ scheme and the $\EXE$ scheme show temporal
convergence with order $\sfrac{1}{2}$, the $\SESRK$ scheme attains
order $1$ and the $\EWP$ scheme and the $\RKTNV$ scheme both show order
$\sfrac{3}{2}$ in the numerical simulations.
\section{Proofs} \label{Sec:proofs}
In Section~\ref{Sec:Bound}, we first prove the boundedness of 
the approximation process. This in turn is essential for the 
analysis of convergence in Section~\ref{Sec:Convergence}.
%
%
Note that since $B$ is a pointwise multiplicative Nemytskii-type 
operator, the following specific first and second commutativity conditions
\begin{equation} \label{eq:comm_con_1}
	\big( B'(v) (B(\tilde{v}) u) \big) \tilde{u} =
	\big( B'(v) (B(\tilde{v}) \tilde{u}) \big) u
\end{equation}
for all $v, \tilde{v} \in H_{\beta}$, $u, \tilde{u} \in U_0$  as well as
\begin{equation} \label{eq:comm_con_2}
	\begin{split}
		&\big( B'(v) \big( B'(\tilde{v}) (B(\hat{v}) u_1) \big) u_2 \big) u_3
		= \big( B'(v) \big( B'(\tilde{v}) ( B(\hat{v}) u_{\pi(1)} ) \big) u_{\pi(2)} 
		\big) u_{\pi(3)} ,\\
		&\big( B''(w) \big( B(\tilde{w} ) u_1, B(\hat{w}) u_2 \big) \big) u_3 
		=  \big( B''(w) \big( B(\tilde{w}) u_{\pi(1)}, B(\hat{w}) u_{\pi(2)} \big)
		\big) u_{\pi(3)}
	\end{split}
\end{equation}
for all $v, \tilde{v}, \hat{v}, w, \tilde{w}, \hat{w} \in H$,
$u_1,u_2,u_3 \in U_0$ and any permutation $\pi \in S_3$ are 
fulfilled. We require these conditions in the proofs in 
Section~\ref{Sec:Bound} and Section~\ref{Sec:Convergence}.
Before we prove Proposition~\ref{Prop:Bounds} and Theorem~\ref{Prop:Conv}, 
we rewrite the numerical schemes in closed form and reformulate them slightly. 
First, we consider two terms that are involved in the family of $\RKTN$ 
schemes~\eqref{RKTN-scheme}--\eqref{RKTN-scheme-stages} given by 
Table~\ref{SRK-scheme-Alg-Coeff}. Due to It\^{o}'s formula we have
\begin{equation}\label{eq:Ito_rewrite}
	\int_{mh}^{(m+1)h} (W_s - W_{mh}) \, \mathrm{d}s 
	= \int_{mh}^{(m+1)h} ((m+1)h-s) \, \mathrm{d}W_s
\end{equation}
for all $m\in\{0,\ldots,M-1\}$ and therefore we can rewrite the 
following two terms
\begin{equation} \label{eq:Ito_rewrite_7}
	A \, b(\cdot,Y_m) \Big( \int_{mh}^{(m+1)h} ( W_s - W_{mh} ) \, \mathrm{d}s
	- \frac{h}{2} \Delta W_m \Big)
	= A \, b(\cdot,Y_m) \int_{mh}^{(m+1)h} ((m+\tfrac{1}{2})h-s) \, \mathrm{d}W_s
\end{equation}
and
\begin{multline} \label{eq:Ito_rewrite_8}
	\frac{1}{\cacht} \Big( b(\cdot,Y_m + \cacht [A Y_m + f(\cdot,Y_m)]) 
	- b(\cdot,Y_m) \Big) \Big(  h \Delta W_m - \int_{mh}^{(m+1)h} (W_s - W_{mh} )
	\, \mathrm{d}s \Big) \\
	= \frac{1}{\cacht} \Big( b(\cdot,Y_m + \cacht [A Y_m + f(\cdot,Y_m)])
	- b(\cdot,Y_m) \Big) \int_{mh}^{(m+1)h} (s-mh) \, \mathrm{d}W_s
\end{multline}
for some $\cacht \in\mathbb{R}\setminus\{0\}$.
Note that we base our proof on a more general form of the exponential 
stochastic Runge-Kutta type schemes which involve coefficients 
$\cdrei, \cvier, \cfuenf, \cacht, \cneun, \czehn, \celf, \aelf \in\mathbb{R}\setminus\{0\}$ 
that may depend on $h$ and thus are more flexible 
(in the following called generalized $\RKTN$ schemes). Then, we show that these
schemes converge with the desired order if we choose the 
coefficients as we did in the derivation of the exponential 
stochastic Runge-Kutta type schemes stated in Section~\ref{Sec:DFM}, 
see Remark~\ref{Rem:CoeffChoice}. One can easily check that the 
schemes are identical in this case.
Here, we also replace \eqref{eq:Ito_rewrite}, \eqref{eq:Ito_rewrite_7} and 
\eqref{eq:Ito_rewrite_8} in the numerical schemes and insert the schemes
iteratively in order to obtain the expression for the approximation process
\begin{subequations} \label{eq:DF_summed}
	\begin{align}
		(\textcolor{black}{\ref{eq:DF_summed}}) \quad \quad 
		Y_m &= \e^{Amh} X_0 + \sum_{l=0}^{m-1} 
		\e^{A(m-l-\tfrac{1}{2})} \bigg\{ h f(\cdot,Y_l) \notag \\
		&\quad + \frac{h^2}{2\cdrei} 
		\Big( f(\cdot,Y_l+\cdrei [A Y_l + f(\cdot,Y_l)]) - f(\cdot,Y_l) \Big) \label{eq:DQsum_3} \\
		&\quad + \frac{1}{\cvier} 
		\Big( f(\cdot, Y_l + \cvier b(\cdot,Y_l)) - f(\cdot,Y_l) \Big) \WdsOK \label{eq:DQsum_4} \\
		&\quad + \frac{h^2}{4\cfuenf^2} 
		\Big( f(\cdot,Y_l+\cfuenf b(\cdot,Y_l)) - 2 f(\cdot,Y_l) 
		+ f(\cdot,Y_l-\cfuenf b(\cdot,Y_l)) \Big) \sumJ g_j^2 \label{eq:DQsum_5} \\
		&\quad + b(\cdot,Y_l) \Delta W_l \notag \\
		&\quad + A \, b(\cdot,Y_l) \int_{lh}^{(l+1)h} ((l+\tfrac{1}{2})h-s) \dd W_s \notag \\
		&\quad + \frac{1}{\cacht} 
		\Big( b(\cdot,Y_l + \cacht [A Y_l + f(\cdot,Y_l)]) - b(\cdot,Y_l) \Big) 
		\int_{lh}^{(l+1)h} (s-lh) \dd W_s \label{eq:DQsum_8} \\
		&\quad + \frac{1}{2\cneun} 
		\bigg( b(\cdot,Y_l+\cneun b(\cdot,Y_l)) - b(\cdot,Y_l) \bigg) 
		\Big( \Delta W_l \Big)^2 \label{eq:DQsum_9} \\
		&\quad + \frac{1}{6\czehn^2} 
		\bigg( b(\cdot,Y_l+\czehn b(\cdot,Y_l)) - 2 b(\cdot,Y_l) 
		+ b(\cdot,Y_l-\czehn b(\cdot,Y_l)) \bigg) 
		\Big( \Delta W_l\Big)^3 \label{eq:DQsum_10} \\
		&\quad + \frac{1}{6\aelf} 
		\bigg( b\Big(\cdot,Y_l + \tfrac{\aelf}{\celf} 
		\Big[ b(\cdot,Y_l + \celf b(\cdot,Y_l)) - b(\cdot,Y_l) \Big]\Big) 
		- b(\cdot,Y_l)\bigg)\Big( \Delta W_l\Big)^3 \label{eq:DQsum_11} \\
		&\quad - \frac{h}{2} \frac{1}{\cneun} 
		\Big( b(\cdot,Y_l +\cneun b(\cdot,Y_l)) - b(\cdot,Y_l) \Big) 
		\sumJ g_j^2 \label{eq:DQsum_12} \\
		&\quad - \frac{1}{2} \frac{1}{\czehn^2} 
		\Big( b(\cdot,Y_l +\czehn b(\cdot,Y_l)\big) 
		- 2 b(\cdot,Y_l) + b(\cdot,Y_l-\czehn b(\cdot,Y_l)) \Big) \notag \\
		&\quad \quad \quad \times \sumJ g_j^2 \WdsOK \label{eq:DQsum_13} \\
		&\quad - \frac{h}{2\aelf} \bigg( b\Big( \cdot,Y_l + \tfrac{\aelf}{\celf}
		\Big[ b(\cdot,Y_l+\celf b(\cdot,Y_l)) 
		- b(\cdot,Y_l)\Big]\Big) -b(\cdot,Y_l)\bigg) \notag \\
		&\quad \quad \quad \times \sumJ g_j^2 \Delta W_l \bigg\} \label{eq:DQsum_14} 
	\end{align}
\end{subequations}
for $m \in \{0,1, \ldots, M\}$.
In the next subsection, we prove the boundedness of this approximation process.
\subsection{Proof of a uniform bound for approximations} 
\label{Sec:Bound}
\begin{proof}[Proof of Proposition~\ref{Prop:Bounds}] 
%
%
%
%
In the first part of the proof, we employ Taylor expansions at the 
pointwise level in such a way that the terms in approximation~\eqref{eq:DF_summed} 
can be expressed in terms of the mappings $F$ and $B$.  In this way, we rewrite 
the generalized $\RKTN$ schemes in the form that we will analyze. 
\subsubsection*{Taylor expansions}
Let $z \in \DomX$. We consider the terms in~\eqref{eq:DF_summed} separately.
For \eqref{eq:DQsum_3}, we use the Taylor expansion
\begin{multline}\label{eq:Taylor_bound_3}
    \frac{1}{\cdrei} \Big( f\big(z,Y_l(z)
    + \cdrei [A Y_l(z)+f(z,Y_l(z))]\big) - f(z,Y_l(z)) \Big) \\
    = \int_0^1 F'(Y_l+r \cdrei [A Y_l + F(Y_l)]) \dd r 
    \Big( A Y_l + F(Y_l) \Big) (z).  
\end{multline}
Term \eqref{eq:DQsum_4} can be expressed as
\begin{multline}\label{eq:Taylor_bound_4}
    \frac{1}{\cvier} \Big( f\big(z,Y_l(z)+ \cvier b(z,Y_l(z))\big) 
    - f(z,Y_l(z)) \Big) \Wds (z) \\
    = \int_0^1 F'(Y_l+r \cvier b(\cdot,Y_l)) \dd r 
    \bigg( B(Y_l) \WdsOK \bigg) (z)
\end{multline}
and the Taylor expansion for \eqref{eq:DQsum_5} reads as
\begin{multline}\label{eq:Taylor_bound_5}
    \frac{1}{\cfuenf^2} \Big( f\big(z,Y_l(z)
    + \cfuenf b(z,Y_l(z))\big) - 2 f(z,Y_l(z)) + f\big(z,Y_l(z)
    - \cfuenf b(z,Y_l(z))\big) \Big) g_j^2(z) \\
    = \int_0^1 \Big[ F''(Y_l+r \cfuenf b(\cdot,Y_l)) 
    + F''(Y_l-r \cfuenf b(\cdot,Y_l))\Big] (1-r) \dd r 
    \Big( B(Y_l) g_j, B(Y_l) g_j\Big) (z).
\end{multline}
Next, we reformulate \eqref{eq:DQsum_8} as
\begin{multline}\label{eq:Taylor_bound_8}
    \frac{1}{\cacht} \Big( b(z,Y_l(z) 
    + \cacht [A Y_l(z) + f(z,Y_l(z))]) - b(z,Y_l(z)) \Big) 
    \bigg( \int_{lh}^{(l+1)h} (s-lh) \dd W_s \bigg) (z) \\
    = \int_0^1 B'(Y_l + r \cacht [A Y_l + F(Y_l)]) \dd r 
    \Big( A Y_l + F(Y_l) \Big) 
    \bigg( \int_{lh}^{(l+1)h} (s-lh) \dd W_s \bigg) (z).
\end{multline}

The following Taylor expansions are used twice with varying 
$u, \tilde{u} \in U_0$. 
In~\eqref{eq:DQsum_9}, we use $u=\Delta W_l$ and 
in~\eqref{eq:DQsum_12}, we employ $u=g_j$. In both cases, 
we work with the expansion 
\begin{multline}\label{eq:Taylor_bound_912}
    \frac{1}{\cneun} \Big( b\big( z,Y_l(z)+ \cneun b(z,Y_l(z))\big) 
    - b(z,Y_l(z)) \Big) u^2(z) \\
    = \bigg( \int_0^1 B'(Y_l+r \cneun b(\cdot,Y_l)) \dd r 
    \Big( B(Y_l)u \Big) u \bigg) (z).
\end{multline}
Then, we consider Taylor expansions in~\eqref{eq:DQsum_10} 
with $u=\tilde{u}=\Delta W_l$ and in~\eqref{eq:DQsum_13} 
with $u=g_j$ and $\tilde{u}=\WdsOK$. We write
\begin{multline}\label{eq:Taylor_bound_1013}
    \frac{1}{\czehn^2} \Big( b\big(z,Y_l(z)+\czehn b(z,Y_l(z))\big) 
    - 2 b(z,Y_l(z)) + b\big(z,Y_l(z)-\czehn b(z,Y_l(z))\big) \Big) 
    u^2(z) \tilde{u}(z) \\
    = \bigg( \int_0^1 \Big[ B''(Y_l+r \czehn b(\cdot,Y_l)) 
    + B''(Y_l-r \czehn b(\cdot,Y_l))\Big] (1-r) \dd r 
    \Big( B(Y_l)u, B(Y_l)u \Big) \tilde{u} \bigg) (z).
\end{multline}
Finally, we consider the last two terms.
We compute the Taylor expansion in~\eqref{eq:DQsum_11} with 
$u=\tilde{u}=\Delta W_l$ and in~\eqref{eq:DQsum_14} with $u=g_j$ 
and $\tilde{u}=\Delta W_l$. This results in
\begin{multline}\label{eq:Taylor_bound_1114}
	\frac{1}{\aelf} \bigg( b\big( z,Y_l(z)
	+ \tfrac{\aelf}{\celf}\big[ b(z,Y_l(z)
	+ \celf b(z,Y_l(z))) -b(z,Y_l(z))\big]\big) 
	- b(z,Y_l(z)) \bigg) u^2(z) \tilde{u}(z) \\
    \shoveleft{=\frac{1}{\celf} \int_0^1 b_y \big(z,Y_l(z)
    	+r \tfrac{\aelf}{\celf} 
    	\big[ b(z,Y_l(z)+\celf b(z,Y_l(z)))-b(z,Y_l(z))\big]\big)}  \\
    \shoveright{\Big[ b(z,Y_l(z)+\celf b(z,Y_l(z)))
    	-b(z,Y_l(z))\Big] \dd r 
    	\; u^2(z) \tilde{u}(z)} \\
	\shoveleft{=\int_0^1 b_y\big(z,Y_l(z)+r \tfrac{\aelf}{\celf} 
		\big[ b(z,Y_l(z)+\celf b(z,Y_l(z)))
		-b(z,Y_l(z))\big]\big) \dd r}  \\
	\shoveright{\int_0^1 b_y( z,Y_l(z)
		+ r \celf b(z,Y_l(z))) b(z,Y_l(z)) \dd r 
		\; u^2(z) \tilde{u}(z)}  \\
	\shoveleft{=\Big( \int_0^1 B'\big( Y_l 
		+ r \tfrac{\aelf}{\celf} \big[ b(\cdot,Y_l+\celf b(\cdot,Y_l))
		-b(\cdot,Y_l)\big]\big) \dd r} \\
	\Big( \int_0^1 B'(Y_l+r\celf b(\cdot,Y_l)) \dd r 
	\Big( B(Y_l) u \Big) u \Big) \tilde{u} \Big) (z).
\end{multline}
\\ \\

%
In our setting of pointwise multiplicative Nemytskii-type operators, the commutativity 
condition stated in~\eqref{eq:comm_con_1} obviously holds. 
Therewith, we can rewrite the expressions
in~\eqref{eq:Taylor_bound_912} in terms of iterated stochastic 
integrals, see Lemma~\ref{Lem:CC1},
\begin{multline} \label{eq:comm_bound_rewrite_912}
	\frac{1}{2} \int_0^1  B'(Y_l+r \cneun b(\cdot,Y_l)) \dd r 
	\Big( B(Y_l)\Delta W_l \Big) \Delta W_l \\
	- \frac{h}{2} \sumJ \int_0^1  B'(Y_l+r \cneun b(\cdot,Y_l)) \dd r 
	\Big( B(Y_l) g_j \Big) g_j \\
	= \int_{lh}^{(l+1)h} \int_0^1  B'(Y_l+r \cneun b(\cdot,Y_l)) \dd r 
	\bigg( \int_{lh}^s B(Y_l) \dd W_u \bigg) \dd W_s.
\end{multline}
Similarly, we can employ the fact that the second commutativity 
condition~\eqref{eq:comm_con_2} is fulfilled and use Lemma~\ref{lem:CC2}.
We get
\begin{align}
	\label{eq:comm_bound_rewrite_10111314}
	&\frac{1}{6} \int_0^1 \Big[ B''(Y_l+r \czehn b(\cdot,Y_l)) 
	+ B''(Y_l -r \czehn b(\cdot,Y_l))\Big] (1-r) \dd r
	\Big( B(Y_l) \Delta W_l, B(Y_l) \Delta W_l \Big) \Delta W_l \nonumber \\
	&\quad +\frac{1}{6} \int_0^1 B' \big( Y_l + r \tfrac{\aelf}{\celf} 
	\big[ b(\cdot,Y_l+\celf b(\cdot,Y_l)) -b(\cdot,Y_l) \big] \big) \dd r \nonumber \\
	& \quad \quad \Big( \int_0^1 B'(Y_l +r\celf b(\cdot,Y_l)) \dd r 
	\Big( B(Y_l) \Delta W_l \Big) \Delta W_l \Big) \Delta W_l \nonumber \\
	&\quad -\frac{1}{2} \sumJ \int_0^1 
	\Big[ B''(Y_l +r \czehn b(\cdot,Y_l)) + B''(Y_l -r \czehn b(\cdot,Y_l)) \Big]
	(1-r) \dd r \nonumber \\
	&\quad \quad \Big( B(Y_l) g_j, B(Y_l) g_j \Big) \WdsOK \nonumber \\
	&\quad -\frac{h}{2} \sumJ \int_0^1 B'\big( Y_l + r \tfrac{\aelf}{\celf} 
	\big[ b( \cdot, Y_l +\celf b(\cdot,Y_l)) -b(\cdot,Y_l) \big] \big) \dd r \nonumber \\
	&\quad \quad  \Big( \int_0^1 
	B'(Y_l +r\celf b(\cdot,Y_l)) \dd r \Big( B(Y_l) g_j \Big) g_j \Big) \Delta W_l \nonumber \\
	&= \frac{1}{2} \int_{lh}^{(l+1)h} \int_0^1 
	\Big[ B''(Y_l +r \czehn b(\cdot,Y_l)) + B''(Y_l-r \czehn b(\cdot,Y_l))\Big] 
	(1-r) \dd r \nonumber \\
	&\quad \quad \bigg( \int_{lh}^s B(Y_l) \dd W_u, \int_{lh}^s B(Y_l) \dd W_u \bigg)
	\dd W_s \nonumber \\
	&\quad + \int_{lh}^{(l+1)h} \int_0^1 
	B'\big( Y_l + r \tfrac{\aelf}{\celf} 
	\big[ b(\cdot,Y_l +\celf b(\cdot,Y_l)) -b(\cdot,Y_l) \big] \big) \dd r \nonumber \\
	&\quad \quad  \bigg( \int_{lh}^s \int_0^1 B'(Y_l +r\celf b(\cdot,Y_l)) \dd r 
	\bigg( \int_{lh}^u B(Y_l) \dd W_v \bigg) \dd W_u \bigg) \dd W_s .
\end{align}

Now, we employ the Taylor expansions~\eqref{eq:Taylor_bound_3}--\eqref{eq:Taylor_bound_1114} 
and rewrite the generalized $\RKTN$ schemes~\eqref{eq:DF_summed} 
with the help of~\eqref{eq:comm_bound_rewrite_912} 
and~\eqref{eq:comm_bound_rewrite_10111314}. This leads to 
\begin{equation*}
	Y_m = \e^{Amh} Y_0 + \sum_{l=0}^{m-1} \e^{A(m-l-\tfrac{1}{2})h} 
	\, \hat{\Upsilon}_l^F 
	+ \int_0^{mh} \sum_{l=0}^{m-1} \ind_{[lh,(l+1)h)}(s) 
	\, \e^{A(m-l-\tfrac{1}{2})h} \, \hat{\Upsilon}_l^B(s) \, \mathrm{d}W_s
\end{equation*}
where
\begin{align*}
	\hat{\Upsilon}_l^F &= h F(Y_l) 
	+ \frac{h^2}{2} \int_0^1 F'(Y_l + r \cdrei [A Y_l + F(Y_l)]) \dd r 
	\Big( A Y_l + F(Y_l) \Big) \\
	&\quad + \int_0^1 F'(Y_l+r \cvier b(\cdot,Y_l)) \dd r 
	\bigg( B(Y_l) \WdsOK \bigg) \\
	&\quad + \frac{h^2}{4} \sumJ \int_0^1 \Big[ F''(Y_l+r \cfuenf b(\cdot,Y_l)) 
	+ F''(Y_l-r \cfuenf b(\cdot,Y_l))\Big] (1-r) \dd r 
	\Big( B(Y_l) g_j, B(Y_l) g_j\Big) 
\end{align*}
and 
\begin{align*}
	\hat{\Upsilon}_l^B(s) &= B(Y_l) + A B(Y_l) ((l+\tfrac{1}{2})h-s) \\
	&\quad + \int_0^1 B'(Y_l + r \cacht [AY_l + F(Y_l)]) \dd r 
	\Big( A Y_l + F(Y_l) \Big) (s-lh) \\
	&\quad + \int_0^1  B'(Y_l+r \cneun b(\cdot,Y_l)) \dd r 
	\bigg( \int_{lh}^s B(Y_l) \dd W_u \bigg) \\ 
	&\quad + \frac{1}{2} \int_0^1 \Big[ B''(Y_l+r \czehn b(\cdot,Y_l)) 
	+ B''(Y_l-r \czehn b(\cdot,Y_l))\Big] (1-r) \dd r \\
	&\quad \quad \bigg( \int_{lh}^s B(Y_l) \dd W_u, 
		\int_{lh}^s B(Y_l) \dd W_u \bigg) \\
	&\quad + \int_0^1 B'\big( Y_l + r \tfrac{\aelf}{\celf} 
	\big[ b(\cdot,Y_l+\celf b(\cdot,Y_l))-b(\cdot,Y_l)\big]\big) \dd r \\
	&\quad \quad \bigg( \int_{lh}^s 
		\int_0^1 B'(Y_l+r\celf b(\cdot,Y_l)) \dd r 
		\bigg( \int_{lh}^u B(Y_l) \dd W_v \bigg) \dd W_u \bigg) .
\end{align*}
We work with this expression in order to prove the boundedness 
of the moments. In the following, we estimate the terms on the right
hand side of
\begin{align*}
	\normLpH{p}{\gamma}{Y_m}^2 &\leq 3 \normALpH{\big}{p}{\gamma}{\e^{Amh} Y_0}^2 
	+ 3 \normALpH{\Big}{p}{\gamma}{\sum_{l=0}^{m-1} 
		\e^{A(m-l-\tfrac{1}{2})h} \hat{\Upsilon}_l^F}^2 \\
	&\quad + 3 \normALpH{\Big}{p}{\gamma}{\int_0^{mh} \sum_{l=0}^{m-1} 
		\ind_{[lh,(l+1)h)}(s) \, \e^{A(m-l-\tfrac{1}{2})h}  \hat{\Upsilon}_l^B(s) \dd W_s}^2 
\end{align*}
separately in order to prove the assertion.
%
%
%
%
%
%
\subsubsection{Preliminaries}
\label{Sec:Preliminaries}
In this section, we give some estimates that are needed throughout the proof. 
First of all, note that from the assumptions it follows that $\alpha>0$, 
$\beta \in (\tfrac{1}{2}, \tfrac{3}{2})$, 
$0< \delta \leq \beta \leq \gamma$, 
$\beta-\gamma \leq 0$, $\delta-\beta \leq 0$, $\alpha-\gamma \leq 0$ and
$\gamma-\beta \in [0, \tfrac{1}{2})$. 

For the proof of bounded moments of the approximation process, 
the following estimates, denoted as (P1)--(P4), will be used in the sequel.
For expressions (P1)--(P3), we employ the linear growth of $B$ in the 
$H_{\beta}$-Norm as well as the properties of operator $A$. Estimate (P4)~is 
then obtained with the help of the linear growth of $F$ in the 
$H_{\alpha}$-Norm. Thus, it holds that
\begin{align*}
	&\text{(P1)} \quad &\normLpHS{p}{\beta}{B(Y_l)} 
    &\leq \normALpR{\Big}{1+\etaA{\beta-\gamma} \norm{Y_l}_{H_{\gamma}}} \\
    & & &\leq C(K) \Big( 1 + \normLpH{p}{\gamma}{Y_l} \Big) \\
	&\text{(P2)} \quad &\normLpHS{p}{\delta}{B(Y_l)} 
    &\leq \etaA{\delta-\beta} \normALpR{\Big}{1+\etaA{\beta-\gamma} \norm{Y_l}_{H_{\gamma}}} \\
    & & &\leq C(K) \Big( 1 + \normLpH{p}{\gamma}{Y_l} \Big) \\ 
	&\text{(P3)} \quad &\normLpHS{p}{}{B(Y_l)}
    &\leq \etaA{-\frac{\beta}{2}} 
    \normALpR{\Big}{1+\etaA{\frac{\beta}{2}-\gamma} \norm{Y_l}_{H_{\gamma}}} \\
	& & &\leq C(K) \Big( 1 + \normLpH{p}{\gamma}{Y_l} \Big) \\
	&\text{(P4)} \quad &\normLpH{p}{}{AY_l+F(Y_l)}
    &\leq \Big( \etaA{1-\gamma} + \eta \etaA{-\gamma/2}^2 \Big) \normLpH{p}{\gamma}{Y_l} \\
	& & & \quad + \etaA{-\frac{\alpha}{2}}^2 \normLpR{1+\etaA{\alpha-\gamma} \norm{Y_l}_{H_{\gamma}}} \\
	& & &\leq C(K) \Big( 1 + \normLpH{p}{\gamma}{Y_l} \Big)
\end{align*}
where we note that
$\beta-\delta \in [0,1)$, 
$\frac{\beta}{2} \in (\frac{\gamma}{2}-\frac{1}{4}, \frac{\gamma}{2}]
\subseteq (\frac{1}{4}, \frac{3}{4})$, 
$\gamma-\frac{\beta}{2}\in[\frac{\gamma}{2},\frac{\gamma}{2}+\frac{1}{4})\subset [\frac{1}{2},1)$,
$\gamma-1 \in [0,\frac{1}{2})$,
$\frac{\gamma}{2} \in [\frac{1}{2}, \frac{3}{4})$, 
$\frac{\alpha}{2} \in(\frac{\gamma-1}{2}, \frac{\gamma}{2}] \subseteq (0,\frac{3}{4})$
and
$\gamma-\alpha \in [0,1)$.
%
%
%
%
%
\subsubsection{Expressions containing the mapping $F$}
\label{Sec:F}
Let $m \in \{ 1, \ldots, M \}$.
By the assumption on the semigroup and with the Cauchy-Schwarz inequality, 
we compute
\begin{align*}
	&\normALpH{\Big}{p}{\gamma}{\sum_{l=0}^{m-1} \e^{A(m-l-\tfrac{1}{2})h}
	\, \hat{\Upsilon}_l^F}^2 \\
	&\leq \bigg( \sum_{l=0}^{m-1} \etaAe{\theta}
	\normALpH{\big}{p}{\gamma-\theta}{\e^{\tfrac{A}{2}(m-l-\tfrac{1}{2})h}
	\, \hat{\Upsilon}_l^F} \bigg)^2 \\
	&\leq \bigg( \sum_{l=0}^{m-1} \big(
	\tfrac{1}{2} (m-l-\tfrac{1}{2}) h\big)^{-\theta} h \bigg)
	\bigg( \sum_{l=0}^{m-1} \big(\tfrac{1}{2} (m-l-\tfrac{1}{2})
	h \big)^{-\theta} h^{-1} \,
	\normALpH{\big}{p}{\gamma-\theta}{\e^{\tfrac{A}{2}(m-l-\tfrac{1}{2})h}
	\, \hat{\Upsilon}_l^F}^2 \bigg).
\end{align*}
For now, we drop the sum and concentrate on the summand. 
Taking into account that $\gamma -\theta -\alpha \leq 0$, we use the 
properties of the semigroup again. In the second step, we 
employ~\eqref{eq:Ito_rewrite} to rewrite the stochastic integral
and we get
\begin{align*}
     &\normALpH{\Big}{p}{\gamma-\theta}{\e^{\tfrac{A}{2}(m-l-\tfrac{1}{2})h} \hat{\Upsilon}_l^F}^2 \\
     &\leq 4 h^2 \etaA{\gamma-\theta-\alpha}^2 
     \normA{\Big}{\e^{\tfrac{A}{2}(m-l-\tfrac{1}{2})h}}_{L(H)}^2 \normLpH{p}{\alpha}{F(Y_l)}^2 \\
     &\quad + 4 \etaAe{\gamma-\theta}^2 \\
     &\qquad  \cdot \bigg( \frac{h^4}{4} \normLpH{p}{}{\int_0^1 F'(Y_l + r \cdrei [A Y_l + F(Y_l)]) \dd r 
     	\Big( A Y_l + F(Y_l) \Big)}^2 \\
     &\quad +\bigg\lVert \int_0^1 F'(Y_l+r \cvier b(\cdot,Y_l)) \dd r \bigg( B(Y_l) \WdsOK
      \bigg) \bigg\rVert_{L^p(\Omega;H)}^2 \\
     &\quad +\frac{h^4}{16} \bigg\lVert \sumJ \int_0^1 \Big[ F''(Y_l+r \cfuenf b(\cdot,Y_l)) 
     	+ F''(Y_l-r \cfuenf b(\cdot,Y_l))\Big] (1-r) \dd r \\
     &\quad \qquad \Big( B(Y_l) g_j, B(Y_l) g_j\Big)\bigg\rVert_{L^p(\Omega;H)}^2 \bigg) \\
     &\leq C(K) h^2 \normLpH{p}{\alpha}{F(Y_l)}^2 
     +C(K) ( \tfrac{1}{2}
     (m-l-\tfrac{1}{2})h)^{-2(\gamma -\theta)} \\
     & \qquad  \cdot \bigg( h^4 \normALpR{\bigg}{\int_0^1 \norm{F'(Y_l + r \cdrei [A Y_l + F(Y_l)])}_{L(H)} \dd r 
     	\norm{A Y_l + F(Y_l)}_H}^2 \\
     &\quad +\normALpR{\bigg}{\int_0^1 \norm{F'(Y_l+r \cvier b(\cdot,Y_l))}_{L(H)} \dd r 
     	\normA{\Big}{\int_{lh}^{(l+1)h} ((l+1)h-s) B(Y_l) \dd W_s}_H}^2 \\
     &\quad + h^4 \Big\lVert \int_0^1 \norm{F''(Y_l+r \cfuenf b(\cdot,Y_l)) 
     	+ F''(Y_l-r \cfuenf b(\cdot,Y_l))}_{L^{(2)}(H,H)} (1-r) \dd r \\ 
     &\qquad  \cdot \norm{B(Y_l)}_{HS(U_0,H)}^2 \Big\rVert_{L^p(\Omega; \mathbb{R} )}^2 \bigg).
\end{align*}
Next, we make use of the boundedness of $F'(v)$ in $L(H)$ and 
the linear growth of the product of $F''$ and $B$. This yields
\begin{align*}
     &\normALpH{\Big}{p}{\gamma-\theta}{\e^{\tfrac{A}{2}(m-l-\tfrac{1}{2})h}
     \hat{\Upsilon}_l^F}^2 \\
     &\leq C(K) h^2 \normLpH{p}{\alpha}{F(Y_l)}^2
     + C(K) (\tfrac{1}{4} h)^{-2(\gamma -\theta)}
     \bigg( h^4 \normLpH{p}{}{A Y_l + F(Y_l)}^2 \\
     &\quad +\normALpH{\Big}{p}{}{\int_{lh}^{(l+1)h} ((l+1)h-s) B(Y_l) 
     \dd W_s}^2 \\
     &\quad + h^4 \normALpR{\Big}{\int_0^1 2 \Big( 1  +  \norm{Y_l}_H 
     + r \, | \cfuenf |  \norm{b(\cdot, Y_l)}_H\Big) (1-r) \dd r}^2 \bigg).
\end{align*}
By \cite[Thm 4.37]{MR3236753} and the computations~(P3) and~(P4) 
in Subsection~\ref{Sec:Preliminaries}, we obtain due to 
$\gamma-\theta \in (0,\tfrac{1}{2}]$ that
\begin{align*}
     &\normALpH{\Big}{p}{\gamma-\theta}{\e^{\tfrac{A}{2}(m-l-\tfrac{1}{2})h} \hat{\Upsilon}_l^F}^2 \\
     &\leq C(K) h^2 \normLpH{p}{\alpha}{F(Y_l)}^2
     + C(K) h^{-2(\gamma -\theta)}
     \bigg( h^4 \normLpH{p}{}{A Y_l + F(Y_l)}^2 \\
     &\quad +  \frac{1}{3} h^3 \bigg( \frac{p(p-1)}{2} \bigg)^{\frac{p}{2}}  \normLpHS{p}{}{B(Y_l)}^2
     + h^4 \big( 1 + \| Y_l \|_{L^p(\Omega; H_{\gamma})}  \big)^2 \bigg) \\     
     &\leq C(K) h^2 \, \big( 1 + h^{2-2(\gamma -\theta)} + h^{1 -2(\gamma -\theta)} \big) 
     \Big( 1 + \normLpH{p}{\gamma}{Y_l}^2 \Big) \\
     &\leq C(K) h^2 \, \Big( 1 + \normLpH{p}{\gamma}{Y_l}^2 \Big).
\end{align*}
%
%
%
%
%
%
\subsubsection{Expressions containing operator $B$}
\label{Sec:B}
In this section, again we fix some $m \in \{1, \ldots, M\}$. We use
\cite[Thm 4.37]{MR3236753} and the properties of the indicator function 
to obtain due to $\gamma-\frac{\theta}{2} \in (\frac{1}{2},1)$ and
$\gamma-\frac{\theta}{2} \leq \beta$ that 
\begin{align*}
	&\normALpH{\bigg}{p}{\gamma}{\int_0^{mh} \sum_{l=0}^{m-1} \ind_{[lh,(l+1)h)}(s) 
		\, \e^{A(m-l-\tfrac{1}{2})h} \hat{\Upsilon}_l^B (s) \dd W_s}^2 \\ 
	&= \normALpH{\bigg}{p}{}{
	(\eta-A)^{\gamma-\frac{\theta}{2}} \sum_{l=0}^{m-1} (\eta-A)^{\frac{\theta}{2}}
	\e^{\frac{A}{2}(m-l-\tfrac{1}{2})h}
	\int_0^{mh} \ind_{[lh,(l+1)h)}(s) \, \e^{\frac{A}{2}(m-l-\tfrac{1}{2})h} 
	\, \hat{\Upsilon}_l^B (s) \dd W_s}^2 \\
	&= \normALpH{\bigg}{p}{\gamma-\theta/2}{
	\int_0^{mh} \sum_{l=0}^{m-1} (\eta-A)^{\frac{\theta}{2}}
	\e^{A(m-l-\tfrac{1}{2})h} \ind_{[lh,(l+1)h)}(s) \, 
	\hat{\Upsilon}_l^B (s) \dd W_s}^2 \\
	&\leq C(K) \int_0^{mh} \normALpHS{\Big}{p}{\gamma-\theta/2}{
	\sum_{l=0}^{m-1} \ind_{[lh,(l+1)h)}(s) \, (\eta-A)^{\frac{\theta}{2}} \,
	\e^{A(m-l-\tfrac{1}{2})h} \, \hat{\Upsilon}_l^B (s) }^2 \dd s \\
	&\leq C(K) \int_0^{mh} \sum_{l=0}^{m-1} \ind_{[lh,(l+1)h)}(s) \, 
	\Big\| (\eta-A)^{\frac{\theta}{2}} \, \e^{A \frac{1}{2}(m-l-\tfrac{1}{2})h}
	\Big\|_{L(H)}^2 \\
	&\quad \cdot \normALpHS{\Big}{p}{\gamma-\theta/2}{
	\e^{A \frac{1}{2}(m-l-\tfrac{1}{2})h} \, \hat{\Upsilon}_l^B (s) }^2 \dd s \\
	&\leq C(K) \sum_{l=0}^{m-1} 2^{\theta} \big( (m-l-\tfrac{1}{2}) h\big)^{-\theta} \int_{lh}^{(l+1)h} \normALpHS{\Big}{p}{\gamma-\theta/2}{\e^{\tfrac{A}{2}(m-l-\tfrac{1}{2})h} \hat{\Upsilon}_l^B (s)}^2 \dd s .
\end{align*}
Using $\gamma-\frac{\theta}{2}-\beta \leq 0$, the integrand on the right hand side can be further estimated as follows
\begin{align*}
	&\normALpHS{\Big}{p}{\gamma-\theta/2}{\e^{\tfrac{A}{2}(m-l-\tfrac{1}{2})h}
   	\hat{\Upsilon}_l^B(s)}^2 \\
    &\leq 6 \bigg[ \etaA{\gamma -\theta/2 -\beta}^2 
    \Big\| \e^{\tfrac{A}{2}(m-l-\tfrac{1}{2})h} \Big\|_{L(H)}^2
    \normALpHS{\big}{p}{\beta}{B(Y_l)}^2 \\
    &\quad + ((l+\tfrac{1}{2})h-s)^2 \etaA{\gamma -\theta/2 -\beta}^2
    \normA{\Big}{A \e^{\tfrac{A}{2}(m-l-\tfrac{1}{2})h}}_{L(H)}^2
    \normALpHS{\big}{p}{\beta}{B(Y_l)}^2 \\
    &\quad + \etaAe{\gamma-\theta/2-\delta}^2 \\
    &\qquad \cdot \normALpHS{\bigg}{p}{\delta}{\int_0^1  B'(Y_l
    +r \cneun b(\cdot,Y_l)) \dd r \bigg( \int_{lh}^s B(Y_l) \dd W_u \bigg)}^2 \\
    &\quad + \etaAe{\gamma-\theta/2}^2 \\
    &\qquad \cdot \bigg( (s-lh)^2 \normALpR{\bigg}{\int_0^1 \norm{B'(Y_l 
    + r \cacht [AY_l + F(Y_l)])}_{L(H,HS(U_0,H))} \dd r \norm{A Y_l + F(Y_l)}_H}^2 \\
    &\quad + \frac{1}{4}
    \bigg\lVert \int_0^1 \norm{B''(Y_l+r \czehn b(\cdot,Y_l)) + B''(Y_l
    -r \czehn b(\cdot,Y_l))}_{L^{(2)}(H,HS(U_0,H))} (1-r) \dd r \\
    &\qquad \cdot  \norm{\int_{lh}^s B(Y_l) \dd W_u}_H^2
    \bigg\rVert_{L^p(\Omega;\mathbb{R})}^2 \\
    &\quad + \bigg\lVert \int_0^1 \norm{B'\big( Y_l + r \tfrac{\aelf}{\celf} 
    \big[ b(\cdot,Y_l+\celf b(\cdot,Y_l))-b(\cdot,Y_l)\big]\big)}_{L(H,HS(U_0,H))} 
	\dd r \\
    &\qquad \cdot  \normA{\bigg}{\int_{lh}^s \int_0^1 B'(Y_l+r\celf b(\cdot,Y_l)) \dd r 
    	\bigg( \int_{lh}^u B(Y_l) \dd W_v \bigg) \dd W_u}_H \bigg\rVert_{L^p(\Omega;\mathbb{R})}^2 \bigg) 
    \bigg] .
\end{align*}
By the properties of the semigroup, due to the boundedness of $B'$ and 
making use of $\gamma-\theta/2-\delta = \max (0, \gamma-\theta/2-\delta) + \min(0, \gamma-\theta/2-\delta)$, 
we compute
\begin{align*}
	&\normALpHS{\Big}{p}{\gamma-\theta/2}{\e^{\tfrac{A}{2}(m-l-\tfrac{1}{2})h}
	\hat{\Upsilon}_l^B(s)}^2 \\
    &\leq C(K) \Big( 1 + ((l+\tfrac{1}{2})h-s)^2 h^{-2} \Big)
    \normALpHS{\big}{p}{\beta}{B(Y_l)}^2 \\
    &\quad + C(K) h^{-2 \, \max(0,\gamma-\frac{\theta}{2}-\delta) }
    \normALpH{\Big}{p}{\delta}{\int_{lh}^s B(Y_l) \dd W_u}^2 \\
    &\quad + C(K) h^{2(\gamma-\theta/2)}
    \bigg( (s-lh)^2 \normALpH{\big}{p}{}{A Y_l + F(Y_l)}^2 \\
    &\quad + \Big\lVert \int_{lh}^s \Big( \int_0^1 \norm{B''(Y_l
    +r \czehn b(\cdot,Y_l)) + B''(Y_l-r \czehn b(\cdot,Y_l))}_{L^{(2)}(H,HS(U_0,H))}
	(1-r) \dd r \Big)^{1/2} \\
	&\qquad  B(Y_l) \dd W_u \Big\rVert_{L^{2p}(\Omega;H)}^4 \\
    &\quad + \normALpH{\Big}{p}{}{\int_{lh}^s \int_0^1 B'(Y_l+r\celf b(\cdot,Y_l)) 
    \dd r \bigg( \int_{lh}^u B(Y_l) \dd W_v \bigg) \dd W_u}^2 \bigg). 
\end{align*}
Now, we apply~\cite[Thm 4.37]{MR3236753} again. Moreover, we
employ the assumption on the product of $B''$ and $B$ to obtain
\begin{align*}
	&\normALpHS{\Big}{p}{\gamma-\theta/2}{\e^{\tfrac{A}{2}(m-l-\tfrac{1}{2})h}
	\hat{\Upsilon}_l^B(s)}^2 \\
    &\leq C(K) \Big( 1 + ((l+\tfrac{1}{2})h-s)^2 h^{-2} \Big)
    \normALpHS{\big}{p}{\beta}{B(Y_l)}^2 \\
    &\quad + C(K) h^{-2 \, \max(0,\gamma-\theta/2-\delta)  } 
    (s-lh) \normALpHS{\big}{p}{\delta}{B(Y_l)}^2 \\
    &\quad + C(K) h^{2(\gamma-\theta/2)}
    \bigg( (s-lh)^2 \normALpH{\big}{p}{}{A Y_l + F(Y_l)}^2 
    + (s-lh)^2 \\
    &\qquad \cdot \normALpR{\Big}{\int_0^1 \norm{B''(Y_l+r \czehn b(\cdot,Y_l)) 
    + B''(Y_l-r \czehn b(\cdot,Y_l))}_{L^{(2)}(H,HS(U_0,H))} (1-r) \dd r \\
    &\qquad \cdot
	\norm{B(Y_l)}_{HS(U_0,H)}^2}^2 \\
    &\quad + \int_{lh}^s \normALpHS{\Big}{p}{}{\int_0^1 B'(Y_l+r\celf b(\cdot,Y_l))
    \dd r \bigg( \int_{lh}^u B(Y_l) \dd W_v \bigg)}^2 \dd u \bigg) \\
    &\leq C(K) \Big( 1 + ((l+\tfrac{1}{2})h-s)^2 h^{-2} \Big)
    \normALpHS{\big}{p}{\beta}{B(Y_l)}^2 \\
    &\quad + C(K) h^{  -2 \, \max(0,\gamma-\theta/2-\delta) }  
    (s-lh) \normALpHS{\big}{p}{\delta}{B(Y_l)}^2 \\
    &\quad + C(K) h^{2(\gamma-\theta/2)} \bigg( (s-lh)^2 
    \normALpH{\big}{p}{}{A Y_l + F(Y_l)}^2 \\
    &\quad + (s-lh)^2 
    \normALpR{\Big}{\int_0^1 2 \Big( 1 + 2 \norm{Y_l}_H 
    + r  \, | \czehn | 
    \norm{b(\cdot, Y_l)}_H\Big) (1-r) \dd r}^2 \\
    &\quad + \int_{lh}^s \normALpR{\Big}{\int_0^1 \norm{B'(Y_l
    + r \celf b(\cdot,Y_l))}_{L(H,HS(U_0,H))} \dd r \normA{\Big}{\int_{lh}^u 
    B(Y_l) \dd W_v}_H}^2 \dd u \bigg) .
\end{align*}
Using~\cite[Thm 4.37]{MR3236753} once more and making
use of the boundedness of $B'$, we get as an estimate for the considered 
integrand
\begin{align*}
	&\normALpHS{\Big}{p}{\gamma-\theta/2}{\e^{\tfrac{A}{2}(m-l-\tfrac{1}{2})h}
   	\hat{\Upsilon}_l^B(s)}^2 \\
    &\leq C(K) \Big( 1 + ((l+\tfrac{1}{2})h-s)^2 h^{-2} \Big)
    \normALpHS{\big}{p}{\beta}{B(Y_l)}^2 \\
    &\quad + C(K) h^{ -2 \, \max(0,\gamma-\theta/2-\delta)  } 
    (s-lh) \normALpHS{\big}{p}{\delta}{B(Y_l)}^2 \\
    &\quad + C(K) h^{2(\gamma-\theta/2)} \bigg( (s-lh)^2 
    \normALpH{\big}{p}{}{A Y_l + F(Y_l)}^2 \\
    &\quad + (s-lh)^2 
    \Big( 1 + \normLpH{p}{}{Y_l}^2 + | \czehn |^2
    \normLpH{p}{}{b(\cdot,Y_l)}^2 \Big) \\
    &\quad + \int_{lh}^s \int_{lh}^u \normALpHS{\big}{p}{}{B(Y_l)}^2 
    \dd v \dd u \bigg) .
\end{align*}
Taking into account that $1-2 \, \max(0, \gamma-\theta/2-\delta) \geq 0$,
we finally obtain by estimates (P1)--(P4) in Subsection~\ref{Sec:Preliminaries} 
that 
\begin{align*}
	&\int_{lh}^{(l+1)h}
	\normALpHS{\Big}{p}{\gamma-\theta/2}{\e^{\tfrac{A}{2}(m-l-\tfrac{1}{2})h}
	\hat{\Upsilon}_l^B(s)}^2 
	\dd s \\
	&\leq C(K) \int_{lh}^{(l+1)h} \Big( 1 + ((l+\tfrac{1}{2})h-s)^2 h^{-2} \Big) \\
	&\quad + h^{{ -2 \, \max(0, \gamma-\theta/2-\delta) } } 
	(s-lh) + h^{2(\gamma-\theta/2)} \Big( (s-lh)^2 
	+ { (s-lh)^2 + \tfrac{1}{2} } (s-lh)^2 \Big) \dd s \\
	&\qquad \cdot \Big( 1 + \normLpH{p}{\gamma}{Y_l}^2 \Big) \\
	&\leq C(K) \, h \, \Big( 1 + h^{1 
	{ -2 \, \max(0, \gamma-\theta/2-\delta) } } 
	+ h^{{1 + 2} (\gamma-\theta/2)} \Big) 
	\Big( 1 + \normLpH{p}{\gamma}{Y_l}^2 \Big) \\
	&\leq C(K) h \Big( 1 + \normLpH{p}{\gamma}{Y_l}^2 \Big) .
\end{align*}
%
%
%
%
As a last step, we combine our computations in Subsections~\ref{Sec:F} 
and~\ref{Sec:B} to obtain
\begin{align*}
	\normLpH{p}{\gamma}{Y_m}^2 
	&\leq C(K) \normLpH{p}{\gamma}{Y_0}^2 + C(K) \sum_{l=0}^{m-1} 
	\big( (m-l-\tfrac{1}{2}) h\big)^{-\theta} h 
	\Big( 1 + \normLpH{p}{\gamma}{Y_l}^2 \Big) \\
	&\leq C(K) \Big( 1 + \normLpH{p}{\gamma}{Y_0}^2 \Big) + C(K) h^{1-\theta}
	\sum_{l=0}^{m-1} (m-l)^{-\theta} \normLpH{p}{\gamma}{Y_l}^2 .
\end{align*}
Finally, applying the generalized discrete Gronwall lemma~\cite{MR880357} 
and performing computations in an analogous manner as 
in~\cite[Sec 4.1]{MR3534472} completes the proof of 
Proposition~\ref{Prop:Bounds}.
%
%
\end{proof}
%
%
%
%
%
\subsection{Proof of convergence}\label{Sec:Convergence}
\begin{proof}[Proof of Theorem~\ref{Prop:Conv}]
For the proof of the temporal convergence of the approximation defined by
the $\RKTN$ schemes, we consider the auxiliary process $\bar{Y}_m$, 
$m\in\{0,\ldots,M\}$, defined as
\begin{align*}
	\bar{Y}_m &= \e^{Amh} X_0 + \sum_{l=0}^{m-1} \e^{A(m-l-\tfrac{1}{2})h} 
	\bigg\{ h F(Y_l) + \frac{h^2}{2} F'(Y_l) \Big( AY_l+F(Y_l) \Big) \\
	&\quad + F'(Y_l) \Big( B(Y_l) \int_{lh}^{(l+1)h} ( W_s - W_{lh} ) \dd s \Big) 
	+ \sumJ \frac{h^2}{4} F''(Y_l)(B(Y_l)g_j,B(Y_l)g_j) \\
	&\quad + \int_{lh}^{(l+1)h} B(Y_l) \dd W_s 
	+ {A} \int_{lh}^{(l+1)h} ((l+\tfrac{1}{2})h-s) B(Y_l) \dd W_s \\
	&\quad + \int_{lh}^{(l+1)h} \int_{lh}^s B'(Y_l) \big( A Y_l + F(Y_l) \big) 
	\dd u \dd W_s \\
	&\quad + \int_{lh}^{(l+1)h} B'(Y_l) \bigg( \int_{lh}^s B(Y_l) \dd W_u \bigg) \dd W_s \\
	&\quad + \frac{1}{2} \int_{lh}^{(l+1)h} B''(Y_l) \bigg( \int_{lh}^s B(Y_l) \dd W_u,
	\int_{lh}^s B(Y_l) \dd W_u \bigg) \dd W_s \\
	&\quad + \int_{lh}^{(l+1)h} B'(Y_l) \bigg( \int_{lh}^s B'(Y_l) \bigg( \int_{lh}^u B(Y_l) 
	\dd W_v \bigg) \dd W_u \bigg) \dd W_s \bigg\}.
\end{align*}
Roughly speaking, this process equals the exponential Wagner-Platen 
type scheme given in~\eqref{Com-Exp-Wagner-Platen-scheme} with 
the difference that the operators are evaluated at the approximation values
$Y_m$, $m\in\{0,\ldots,M\}$, of the generalized $\RKTN$ schemes.

We split the error estimate such that we can transfer parts
of the proof for the Wagner-Platen type scheme in~\cite{MR3534472} 
to our setting for the first term on the right hand side of
\begin{equation*}
    \normLpH{2}{}{X_{mh}-Y_m}^2 
    \leq 2 \normLpH{2}{}{X_{mh}-\bar{Y}_m}^2 + 2 \normLpH{2}{}{\bar{Y}_m-Y_m}^2.
\end{equation*}
The second term on the right hand side is the error 
that results from replacing the derivatives 
in~\eqref{Com-Exp-Wagner-Platen-scheme} by suitable approximations.
Then, the estimates established in Sections~\ref{sec:convWP} 
and~\ref{sec:convDF} below imply that
\begin{equation*}
    \normLpH{2}{}{X_{mh}-Y_m}^2 
    \leq C(K) h \sum_{l=0}^{m-1} \normLpH{2}{}{X_{lh}-Y_l}^2 
    + C(K)h^{2\gamma} + C(K)h^{2\gamma} .
\end{equation*}
Finally, the application of the discrete Gronwall lemma~\cite{MR880357} yields
\begin{equation*}
    \normLpH{2}{}{X_{mh}-Y_m}^2 
    \leq C(K) h^{2\gamma} + C(K) \sum_{l=0}^{m-1} h^{2\gamma} h \, 
    \exp \Big( C(K) \sum_{j=l+1}^{m-1} h \Big) \leq C(K) h^{2\gamma}
\end{equation*}
and thus completes the proof for
$\normLpH{2}{}{X_{mh}-Y_m} \leq C(K) h^{\gamma}$.
\subsubsection{Estimates for $\| X_{mh}-\bar{Y}_m \|_{L^2(\Omega;H)}^2$}
\label{sec:convWP}
For this part of the proof, we employ analogous steps as in the corresponding 
proof for the exponential Wagner-Platen type scheme given in~\cite{MR3534472}. 
Let $\lfloor t \rfloor= \max\{s\in \{0,h,\ldots (M-1)h, T\} : s\leq t\}$. 
As in~\cite{MR3534472}, we define the processes
\begin{equation*}
    \hat{X}_t = \e^{A(t-\lfloor t\rfloor)} Y_{\lfloor t\rfloor /h} 
    + \int_{\lfloor t\rfloor}^t \e^{A(t-s)} F(\hat{X}_s) \dd s 
    + \int_{\lfloor t\rfloor}^t \e^{A(t-s)} B(\hat{X}_s) \dd W_s
\end{equation*}
and
\begin{align*}
	\Phi_t &= F(\hat{X}_{\lfloor t\rfloor}) 
	+ F'(\hat{X}_{\lfloor t\rfloor}) 
	\int_{\lfloor t\rfloor}^t \Big( A \hat{X}_{\lfloor t\rfloor}
	+ F(\hat{X}_{\lfloor t\rfloor})\Big) \dd u 
	+ F'(\hat{X}_{\lfloor t\rfloor})\int_{\lfloor t\rfloor}^t B(\hat{X}_{\lfloor t\rfloor}) \dd W_u \\
	&\quad + \frac{1}{2} \sumJ \int_{\lfloor t\rfloor}^t F''(\hat{X}_{\lfloor t\rfloor}) 
	\Big( B(\hat{X}_{\lfloor t\rfloor})g_j, B(\hat{X}_{\lfloor t\rfloor}) g_j\Big) \dd u , \\
    \Psi_t &= B(\hat{X}_{\lfloor t\rfloor})
    + B'(\hat{X}_{\lfloor t\rfloor}) 
    \int_{\lfloor t\rfloor}^t \Big( B(\hat{X}_{\lfloor t\rfloor}) 
    + B'(\hat{X}_{\lfloor t\rfloor}) 
    \int_{\lfloor t\rfloor}^u B(\hat{X}_{\lfloor t\rfloor}) \dd W_v \Big) \dd W_u \\
	&\quad + B'(\hat{X}_{\lfloor t\rfloor}) 
	\int_{\lfloor t\rfloor}^t \Big(A \hat{X}_{\lfloor t\rfloor} 
	+ F(\hat{X}_{\lfloor t\rfloor}) \Big) \dd u \\
	&\quad + \frac{1}{2} B''(\hat{X}_{\lfloor t\rfloor}) 
	\Big( \int_{\lfloor t\rfloor}^t B(\hat{X}_{\lfloor t\rfloor}) \dd W_u, 
	\int_{\lfloor t\rfloor}^t B(\hat{X}_{\lfloor t\rfloor}) \dd W_u \Big) .
\end{align*}
These in turn are used in the definition of the following two
auxiliary processes
\begin{equation*}
 \begin{split}
    Z_t &= \e^{At} X_0 + \int_0^t \e^{A(t-s)} F(\hat{X}_s) \dd s 
    + \int_0^t \e^{A(t-s)} B(\hat{X}_s) \dd W_s , \\
    \hat{Z}_t &= \e^{At} X_0 + \int_0^t \e^{A(t-s)} \Phi_s \dd s 
    + \int_0^t \e^{A(t-s)} \Psi_s \dd W_s.
 \end{split}
\end{equation*}
Due to the moment estimates for the approximations of the generalized $\RKTN$ schemes 
given in Section~\ref{Sec:Bound}, we can directly transfer the biggest
part of the proof in~\cite{MR3534472} to our setting. We only 
detail equation~(62) in~\cite{MR3534472} here, as this is the 
single step that we need to adapt. We compute
\begin{align*}
    &\normLpH{2}{}{X_{mh}-\bar{Y}_m}^2 \\
    &\leq 3 \Big( \normLpH{2}{}{X_{mh}-Z_{mh}}^2 
    + \normALpH{}{2}{}{Z_{mh}-\hat{Z}_{mh}}^2 
    + \normALpH{}{2}{}{\hat{Z}_{mh}-\bar{Y}_{mh}}^2\Big) \\
	&\leq C(K) \sum_{l=0}^{m-1} \int_{lh}^{(l+1)h} \normLpH{2}{}{X_{lh}-Y_l}^2 \dd s 
	+ C(K)h^{2\gamma}
\end{align*}
for all $m\in\{1,\ldots,M\}$.
\subsubsection{Estimates for $\| \bar{Y}_m-Y_m \|_{L^2(\Omega;H)}^2$}
\label{sec:convDF}
Similarly as in the proof of Proposition~\ref{Prop:Bounds}, we 
employ Taylor expansions at the pointwise level in such a way 
that the terms can be expressed in terms of operators. We use 
the same kind of Taylor expansions as before, however, add and 
subtract some terms. In this way, we create all the terms that 
are involved in Lemma~\ref{Lem:CC1} and Lemma~\ref{lem:CC2} presented
in the following,
which in turn are needed to rewrite the schemes in the fashion 
used in the proof. First of all, we start to list the needed Taylor 
expansions.
\subsubsection*{Taylor expansions}
%
We consider the generalized $\RKTN$ schemes given in~\eqref{eq:DF_summed}.
For term~\eqref{eq:DQsum_3}, we use the Taylor expansion
\begin{align}\label{eq:Taylor_conv_3}
    &\frac{1}{\cdrei} \Big( f\big(z,Y_l(z)+ \cdrei [A Y_l(z)+f(z,Y_l(z))]\big) 
    - f(z,Y_l(z)) \Big) \nonumber \\ 
    &= F'(Y_l) \Big( A Y_l + F(Y_l) \Big) (z)  
    + \int_0^1 F'(Y_l+r \cdrei [A Y_l + F(Y_l)]) 
    - F'(Y_l) \dd r \Big( A Y_l + F(Y_l) \Big) (z),
\end{align}
%
instead of~\eqref{eq:DQsum_4} we use
\begin{align}\label{eq:Taylor_conv_4}
	&\frac{1}{\cvier} \Big( f\big(z,Y_l(z)+ \cvier b(z,Y_l(z))\big) 
	- f(z,Y_l(z)) \Big) \Wds (z) \nonumber \\
	&= F'(Y_l) \bigg( B(Y_l) \WdsOK \bigg) (z) \nonumber \\
	&\quad  + \int_0^1 \Big[ F'(Y_l+r \cvier b(\cdot,Y_l)) 
	- F'(Y_l)\Big] \dd r \bigg( B(Y_l) \WdsOK \bigg) (z)
\end{align}
and for~\eqref{eq:DQsum_5}, we work with
\begin{align}\label{eq:Taylor_conv_5}
	&\frac{1}{\cfuenf^2} \Big( f\big(z,Y_l(z)+ \cfuenf b(z,Y_l(z))\big) 
	- 2 f(z,Y_l(z)) + f\big(z,Y_l(z)- \cfuenf b(z,Y_l(z))\big) \Big) g_j^2(z) \nonumber \\
    &= F''(Y_l) \big( B(Y_l) g_j, B(Y_l) g_j \big) (z) \nonumber \\
    &\quad + \int_0^1 \big[ F''(Y_l+r \cfuenf b(\cdot,Y_l)) - 2 F''(Y_l) 
    + F''(Y_l-r \cfuenf b(\cdot,Y_l)) \big] 
    \big( B(Y_l) g_j, B(Y_l) g_j \big) (z) (1-r) \dd r .
\end{align}
%
Further, in~\eqref{eq:DQsum_8} the Taylor expansion 
\begin{align}\label{eq:Taylor_bound_8}
	&\frac{1}{\cacht} \Big( b(z,Y_l(z) + \cacht [A Y_l(z) + f(z,Y_l(z))]) - b(z,Y_l(z)) \Big)
	\bigg( \int_{lh}^{(l+1)h} (s-lh) \dd W_s \bigg)(z) \nonumber \\
    &= B'(Y_l) \big( A Y_l + F(Y_l) \big) 
    \bigg( \int_{lh}^{(l+1)h} (s-lh) \dd W_s \bigg) (z) \nonumber \\
    &\quad + \int_0^1 \big[ B'(Y_l + r \cacht [A Y_l + F(Y_l)]) - B'(Y_l) \big] \dd r 
    \big( A Y_l + F(Y_l) \big) \bigg( \int_{lh}^{(l+1)h} (s-lh) \dd W_s \bigg) (z)
\end{align}
is employed.
The following Taylor expansions are used twice
%
\begin{align}\label{eq:Taylor_conv_912}
	&\frac{1}{\cneun} \Big( b\big( z,Y_l(z)+ \cneun b(z,Y_l(z))\big) 
	- b(z,Y_l(z)) \Big) u^2(z) \nonumber \\
    &= \big( B'(Y_l) \big( B(Y_l) u \big) u \big) (z) 
    + \bigg( \int_0^1  \Big[ B'(Y_l+r \cneun b(\cdot,Y_l)) - B'(Y_l)\Big] \dd r 
    \, \big( B(Y_l)u \big) u \bigg) (z)
\end{align}
with varying 
$u \in U_0$.
To be precise, expression~\eqref{eq:DQsum_9} is replaced with $u=\Delta W_l$ 
and expression~\eqref{eq:DQsum_12} with $u=g_j$.
Moreover, in~\eqref{eq:DQsum_10} we use the Taylor expansion 
\begin{align}\label{eq:Taylor_conv_1013}
	&\frac{1}{\czehn^2} \Big( b\big(z,Y_l(z)+\czehn b(z,Y_l(z))\big) 
	- 2 b(z,Y_l(z)) + b\big(z,Y_l(z)-\czehn b(z,Y_l(z))\big) \Big) 
	u^2(z) \tilde{u}(z) \nonumber \\
    &= \bigg( B''(Y_l)\Big( B(Y_l) u, B(Y_l) u\Big) \tilde{u} \bigg) (z) \nonumber \\
    &\quad + \bigg( \int_0^1 \Big[ B''(Y_l+r \czehn b(\cdot,Y_l))-B''(Y_l)\Big] (1-r) \dd r 
    \Big( B(Y_l)u, B(Y_l)u \Big) \tilde{u} \bigg) (z) \nonumber \\
    &\quad + \bigg( \int_0^1 \Big[ B''(Y_l-r \czehn b(\cdot,Y_l))-B''(Y_l)\Big] (1-r) \dd r 
    \Big( B(Y_l)u, B(Y_l)u \Big) \tilde{u} \bigg) (z)
\end{align}
with $u=\tilde{u}=\Delta W_l$ and 
in expression~\eqref{eq:DQsum_13} with $u=g_j$ and $\tilde{u}=\WdsOK$.
Finally, we employ 
\begin{align}\label{eq:Taylor_conv_1114}
	&\frac{1}{\aelf} \bigg( b\big( z,Y_l(z)
	+\tfrac{\aelf}{\celf}\big[ b(z,Y_l(z)+ \celf b(z,Y_l(z))) -b(z,Y_l(z))\big]\big) 
	- b(z,Y_l(z)) \bigg) u^2(z) \tilde{u}(z) \nonumber \\
	&= \Big( B'(Y_l) \Big( B'(Y_l) \Big( B(Y_l) u \Big) u \Big) \tilde{u} \Big) (z) \nonumber \\
	&\quad + \Big( B'(Y_l) \Big( \int_0^1 B'(Y_l+r\celf b(\cdot,Y_l))-B'(Y_l) \dd r 
	\Big( B(Y_l) u \Big) u \Big) \tilde{u} \Big) (z) \nonumber \\
	&\quad + \Big( \int_0^1 B'\big( Y_l + r \tfrac{\aelf}{\celf} 
	\big[ b(\cdot,Y_l+\celf b(\cdot,Y_l))-b(\cdot,Y_l)\big]\big) - B'(Y_l) \dd r \nonumber \\
	&\quad \quad \Big( \int_0^1 B'(Y_l+r\celf b(\cdot,Y_l)) \dd r 
	\Big( B(Y_l) u \Big) u \Big) \tilde{u} \Big) (z)
\end{align}
in expression~\eqref{eq:DQsum_11} with $u=\tilde{u}=\Delta W_l$ and in~\eqref{eq:DQsum_14} 
with $u=g_j$ and $\tilde{u}=\Delta W_l$.
%

As the commutativity condition~\eqref{eq:comm_con_1} is fulfilled, we can 
rewrite the first terms in~\eqref{eq:Taylor_conv_912} with $u = \Delta W_l$ 
and $u= g_j$ by Lemma~\ref{Lem:CC1}, insert these expressions in~\eqref{eq:DQsum_9}
and~\eqref{eq:DQsum_12}, and obtain
\begin{align}\label{eq:comm_conv_rewrite_912_a}
    &\frac{1}{2} B'(Y_l) \Big( B(Y_l)\Delta W_l \Big) \Delta W_l 
    - \frac{h}{2} \sumJ B'(Y_l) \Big( B(Y_l) g_j \Big) g_j \nonumber \\
    &= \int_{lh}^{(l+1)h} B'(Y_l) \bigg( \int_{lh}^s B(Y_l) \dd W_u \bigg) \dd W_s.
\end{align}
Also by Lemma~\ref{Lem:CC1}, the remaining terms in 
equation~\eqref{eq:Taylor_conv_912} can be expressed as 
\begin{align*} 
	&\frac{1}{2} \int_0^1  \Big[ B'(Y_l+r \cneun b(\cdot,Y_l)) - B'(Y_l)\Big] \dd r 
	\Big( B(Y_l) \Delta W_l \Big) \Delta W_l \\
	&\quad - \frac{h}{2} \sumJ \int_0^1  \Big[ B'(Y_l +r \cneun b(\cdot,Y_l)) - B'(Y_l)\Big] \dd r 
	\Big( B(Y_l) g_j \Big) g_j \\
	&= \int_{lh}^{(l+1)h} \int_0^1  \Big[ B'(Y_l+r \cneun b(\cdot,Y_l)) - B'(Y_l)\Big] 
	\dd r \bigg( \int_{lh}^s B(Y_l) \dd W_u \bigg) \dd W_s.
\end{align*}

We combine equations~\eqref{eq:Taylor_conv_1013} and~\eqref{eq:Taylor_conv_1114} 
according to~\eqref{eq:DF_summed} and as the second commutativity 
condition~(\ref{eq:comm_con_2}) is fulfilled, we get by Lemma~\ref{lem:CC2} 
for the terms in equations~\eqref{eq:Taylor_conv_1013} and~\eqref{eq:Taylor_conv_1114} 
that are involved in the Wagner-Platen scheme that
\begin{align*} 
	&\frac{1}{6} B''(Y_l) \Big( B(Y_l) \Delta W_l, B(Y_l) \Delta W_l \Big) 
	\Delta W_l
	+ \frac{1}{6} B'(Y_l) \Big( B'(Y_l) \Big( B(Y_l) \Delta W_l \Big) 
	\Delta W_l \Big) \Delta W_l \\
	&\quad - \frac{1}{2} \sumJ B''(Y_l) \Big( B(Y_l) g_j, B(Y_l) g_j \Big) \WdsOK \\
	&\quad - \frac{h}{2} \sumJ B'(Y_l) 
	\Big( B'(Y_l) \Big( B(Y_l) g_j \Big) g_j \Big) \Delta W_l \\
	&= \frac{1}{2} \int_{lh}^{(l+1)h} 
	B''(Y_l) \bigg( \int_{lh}^s B(Y_l) \dd W_u, \int_{lh}^s B(Y_l) \dd W_u \bigg) 
	\dd W_s \\
	&\quad +\int_{lh}^{(l+1)h} B'(Y_l) 
	\bigg( \int_{lh}^s B'(Y_l) 
	\bigg( \int_{lh}^u B(Y_l) \dd W_v \bigg) \dd W_u \bigg) \dd W_s .
\end{align*}
Similarly, we get for the rest of terms in equations~\eqref{eq:Taylor_conv_1013} 
and~\eqref{eq:Taylor_conv_1114} that
\begin{align*} 
	&\frac{1}{6} \int_0^1 \Big[ B''(Y_l +r \czehn b(\cdot,Y_l)) -B''(Y_l) \Big]
	(1-r) \dd r \Big( B(Y_l) \Delta W_l, B(Y_l) \Delta W_l \Big) \Delta W_l \\
	&\quad +\frac{1}{6} B'(Y_l) \Big( \int_0^1 B'(Y_l +r\celf b(\cdot,Y_l)) -B'(Y_l) \dd r
	\Big( B(Y_l) \Delta W_l \Big) \Delta W_l \Big) \Delta W_l \\
	&\quad - \frac{1}{2} \sumJ \int_0^1 \Big[ B''(Y_l +r \czehn b(\cdot,Y_l)) -B''(Y_l)\Big]
	(1-r) \dd r \\
	&\quad \quad \Big( B(Y_l) g_j, B(Y_l) g_j \Big) \WdsOK \\
	&\quad - \frac{h}{2} \sumJ B'(Y_l) \Big( \int_0^1 B'(Y_l +r \celf b(\cdot,Y_l)) 
	-B'(Y_l) \dd r \Big( B(Y_l) g_j \Big) g_j \Big) \Delta W_l \\
	&= \frac{1}{2} \int_{lh}^{(l+1)h} \int_0^1 
	\Big[ B''(Y_l +r \czehn b(\cdot,Y_l)) -B''(Y_l) \Big] (1-r) \dd r \\
	&\quad \quad \bigg( \int_{lh}^s B(Y_l) \dd W_u, \int_{lh}^s B(Y_l) \dd W_u \bigg) \dd W_s \\
	&\quad + \int_{lh}^{(l+1)h} B'(Y_l) \bigg( \int_{lh}^s \int_0^1
	B'(Y_l +r \celf b(\cdot,Y_l)) -B'(Y_l) \dd r
	\bigg( \int_{lh}^u B(Y_l) \dd W_v \bigg) \dd W_u \bigg) \dd W_s  
\end{align*}
and that
\begin{align*} 
	&\frac{1}{6} \int_0^1 \Big[ B''(Y_l -r \czehn b(\cdot,Y_l)) -B''(Y_l)\Big] (1-r) \dd r
	\Big( B(Y_l) \Delta W_l, B(Y_l) \Delta W_l \Big) \Delta W_l \\
	&\quad +\frac{1}{6} \int_0^1 B'\big( Y_l + r \tfrac{\aelf}{\celf} 
	\big[ b(\cdot,Y_l+\celf b(\cdot,Y_l))-b(\cdot,Y_l)\big]\big) - B'(Y_l) \dd r \\
	&\quad \quad \Big( \int_0^1 B'(Y_l+r\celf b(\cdot,Y_l)) \dd r 
	\Big( B(Y_l) \Delta W_l \Big) \Delta W_l \Big) \Delta W_l \\
    &\quad - \frac{1}{2} \sumJ \int_0^1 \Big[ B''(Y_l -r \czehn b(\cdot,Y_l))
    -B''(Y_l)\Big] (1-r) \dd r \\
    &\quad \quad \Big( B(Y_l) g_j, B(Y_l) g_j \Big) \WdsOK \\
    &\quad - \frac{h}{2} \sumJ \int_0^1 B'\big( Y_l + r \tfrac{\aelf}{\celf} 
    \big[ b(\cdot,Y_l+\celf b(\cdot,Y_l))-b(\cdot,Y_l)\big]\big) - B'(Y_l) \dd r \\
    &\quad \quad \Big( \int_0^1 B'(Y_l +r \celf b(\cdot,Y_l)) \dd r 
    \Big( B(Y_l) g_j \Big) g_j \Big) \Delta W_l \\
    &= \frac{1}{2} \int_{lh}^{(l+1)h} \int_0^1 
    \Big[ B''(Y_l-r \czehn b(\cdot,Y_l))-B''(Y_l)\Big] (1-r) \dd r \\
	&\quad \quad \bigg( \int_{lh}^s B(Y_l) \dd W_u, \int_{lh}^s B(Y_l) \dd W_u \bigg) \dd W_s \\
    &\quad +\int_{lh}^{(l+1)h} \int_0^1 
    B'\big( Y_l + r \tfrac{\aelf}{\celf} 
    \big[ b(\cdot,Y_l +\celf b(\cdot,Y_l))-b(\cdot,Y_l)\big]\big) - B'(Y_l) \dd r \\
    &\quad \quad \bigg( \int_{lh}^s \int_0^1 B'(Y_l+r\celf b(\cdot,Y_l)) \dd r 
    \bigg( \int_{lh}^u B(Y_l) \dd W_v \bigg) \dd W_u \bigg) \dd W_s ,
\end{align*}
as the commutativity condition~\eqref{eq:comm_con_2} is also fulfilled 
for these terms.

Thus, we can rewrite the approximation as
\begin{equation} \label{eq:Aprox_PsiF_PsiB}
    Y_m = \bar{Y}_m + \sum_{l=0}^{m-1} \e^{A(m-l-\tfrac{1}{2})h} 
    \Big( \Upsilon^F_l + \int_{lh}^{(l+1)h} \Upsilon^B_l(s) \dd W_s \Big)
\end{equation}
where
\begin{align*}
	\Upsilon^F_l &= \frac{h^2}{2} \int_0^1 
	\Big[ F'(Y_l + r \cdrei [A Y_l + F(Y_l)]) - F'(Y_l) \Big] \dd r 
	\Big( A Y_l +F(Y_l) \Big) \\
	&\quad +\int_0^1 \Big[ F'(Y_l+r \cvier b(\cdot,Y_l)) - F'(Y_l)\Big] \dd r 
	\bigg( B(Y_l) \WdsOK \bigg) \\
	&\quad +\frac{h^2}{4} \sumJ \int_0^1 
	\Big[ F''(Y_l+r \cfuenf b(\cdot,Y_l)) - 2 F''(Y_l) 
	+ F''(Y_l-r \cfuenf b(\cdot,Y_l))\Big] (1-r) \dd r \\
	&\quad \quad \big( B(Y_l) g_j, B(Y_l) g_j \big) 
\end{align*}
and 
\begin{align*}
	\Upsilon^B_l(s) &= \int_{lh}^s \int_0^1 \Big[ B'(Y_l+r \cacht [AY_l+F(Y_l)])
	- B'(Y_l)\Big] \dd r \Big( AY_l+F(Y_l)\Big) \dd u \\
	&\quad + \int_0^1  \Big[ B'(Y_l+r \cneun b(\cdot,Y_l)) - B'(Y_l)\Big] \dd r
	\bigg( \int_{lh}^s B(Y_l) \dd W_u \bigg) \\
	&\quad + \frac{1}{2} \int_0^1 \Big[ B''(Y_l+r \czehn b(\cdot,Y_l))-2 B''(Y_l)
	+ B''(Y_l-r \czehn b(\cdot,Y_l))\Big] (1-r) \dd r \\
	&\quad \quad \bigg( \int_{lh}^s B(Y_l) \dd W_u, \int_{lh}^s B(Y_l) 
	\dd W_u \bigg) \\
	&\quad + \int_0^1 B'(Y_l) \bigg( \int_{lh}^s B'(Y_l+r\celf b(\cdot,Y_l))
	- B'(Y_l) \bigg( \int_{lh}^u B(Y_l) \dd W_v \bigg) \dd W_u \bigg) \dd r \\
	&\quad + \int_0^1 B'\big( Y_l + r \tfrac{\aelf}{\celf} \big[ b(\cdot,Y_l 
	+ \celf b(\cdot,Y_l))-b(\cdot,Y_l)\big]\big) - B'(Y_l) \dd r \\
	&\quad \quad \bigg( \int_{lh}^s \int_0^1 B'(Y_l+r\celf b(\cdot,Y_l)) \dd r 
	\bigg( \int_{lh}^u B(Y_l) \dd W_v \bigg) \dd W_u \bigg).
\end{align*}

In the following, we consider the terms on the right hand side 
of the estimate
\begin{align*}
    \normLpH{2}{}{\bar{Y}_m-Y_m}^2
    &= \bigg\lVert \sum_{l=0}^{m-1} \e^{A(m-l-\tfrac{1}{2})h} 
    \Big( \Upsilon^F_l + \int_{lh}^{(l+1)h} \Upsilon^B_l(s) \dd W_s \Big) \bigg\rVert_{L^2(\Omega;H)}^2 \\
    &\leq 2 \bigg( \sum_{l=0}^{m-1} \normA{\Big}{\e^{A(m-l-\tfrac{1}{2})h}}_{L(H)} \normLpH{2}{}{\Upsilon^F_l}\bigg)^2 \\
    &\quad + 2 \, \normALpH{\bigg}{2}{}{\int_0^{mh} \sum_{l=0}^{m-1} \ind_{[lh,(l+1)h)}(s) 
    \, \e^{A(m-l-\tfrac{1}{2})h} \Upsilon^B_l(s) \dd W_s}^2
\end{align*}
separately.
As we need the succeeding basic computations several times, we state
the following lemma.
\begin{lem} \label{Lma:Estimates_FB}
	Given the assumptions in Section~\ref{Sec:Setting} and 
	Theorem~\ref{Prop:Conv}, it holds for $p \geq 2$
	that
	\begin{equation} \label{eq:E_AYFY}
		\big\| AY_l+F(Y_l) \big\|_{L^p(\Omega;H)} \leq C(K)
	\end{equation}
	and for $q \in \{1,2,4\}$, it holds that
	\begin{equation} \label{eq:E_bYBY}
		\Big\| \big\| b(\cdot,Y_l) \big\|_H^2 \,
		\big\| B(Y_l) \big\|_{HS(U_0,H)}^q \Big\|_{L^p(\Omega;\mathbb{R})} \leq C(K) \, .
	\end{equation} 
\end{lem}
\begin{proof}
Note that $\gamma-1\in[0,\frac{1}{2})$ and $\gamma-\alpha\in[0,1)$. 
By the linear growth of $F$ and the properties of the operator $A$, 
we obtain that
\begin{align*} 
	&\normLpH{p}{}{AY_l+F(Y_l)} \nonumber \\
	&\leq \normLpH{p}{}{(\eta - A + \eta)(\eta-A)^{-\gamma}(\eta-A)^{\gamma}Y_l} 
	+ \etaA{-\alpha} \normLpH{p}{\alpha}{F(Y_l)} \nonumber \\
	&\leq \Big( \etaA{1-\gamma} + \eta \etaA{-\gamma/2}^2 \Big) \normLpH{p}{\gamma}{Y_l} 
	+ C(K) \normLpR{1+\norm{Y_l}_{H_{\alpha}}} \nonumber \\
	&\leq C(K) \normLpH{p}{\gamma}{Y_l} + C(K) \Big( 1 + \etaA{\alpha-\gamma} 
	\normLpH{p}{\gamma}{Y_l} \Big) 
	\leq C(K).
\end{align*}
For the second estimate, we point out that $\gamma-\beta \in[0,\frac{1}{2})$.
By the linear growth of $B$, we compute for $q \in \{1,2,4\}$ that
\begin{align*} 
	&\normLpR{\norm{b(\cdot,Y_l)}_H ^{2}\norm{B(Y_l)}_{HS(U_0,H)}^q} \\
	&\leq \etaA{-\beta}^{q} \normLpR{\norm{b(\cdot,Y_l)}_H^{2} \norm{B(Y_l)}_{HS(U_0,H_{\beta})}^q} \\
	&\leq C(K) \normLpR{\big(1 + \etaA{-\gamma/2}^{4} \norm{Y_l}_{H_{\gamma}}^{2} \big)
	\big(1 + \etaA{\beta-\gamma} \norm{Y_l}_{H_{\gamma}} \big)^q} \\
	&\leq C(K) \Big( 1 + \norm{Y_l}_{L^{p}(\Omega;H_{\gamma})}^{2+q}\Big) \leq C(K).
\end{align*} 
\end{proof}
To start with, we treat the terms involving the mapping $F$. 
Let $l\in\{0,\ldots,M-1\}$, 
we employ~\eqref{eq:Ito_rewrite} to rewrite the second term
\begin{align*}
	&\normLpH{2}{}{\Upsilon^F_l} \\ 
	&\leq \frac{h^2}{2} \normLpH{2}{}{\int_0^1 
		\Big[ F'(Y_l + r \cdrei [A Y_l + F(Y_l)]) 
		- F'(Y_l) \Big] \dd r \Big( A Y_l +F(Y_l) \Big)} \\
	&\quad +\normALpH{\bigg}{2}{}{\int_0^1 
		\Big[ F'(Y_l+r \cvier b(\cdot,Y_l)) - F'(Y_l)\Big] \dd r 
		\bigg( B(Y_l) \WdsOK \bigg)} \\
	&\quad +\frac{h^2}{4} \normALpH{\bigg}{2}{}{\sumJ \int_0^1 
		\Big[ F''(Y_l+r \cfuenf b(\cdot,Y_l)) - 2 F''(Y_l) 
		+ F''(Y_l-r \cfuenf b(\cdot,Y_l))\Big] (1-r) \dd r \\
	&\quad \quad \Big( B(Y_l) g_j, B(Y_l) g_j\Big)} \\
	&\leq \frac{h^2}{2} \E{\int_0^1 
		\norm{F'(Y_l + r \cdrei [A Y_l + F(Y_l)]) - F'(Y_l)}_{L(H)}^2 \dd r 
		\norm{A Y_l +F(Y_l)}_H^2}^{1/2} \\
	&\quad +\E{\int_0^1 \norm{F'(Y_l+r \cvier b(\cdot,Y_l)) 
			- F'(Y_l)}_{L(H)}^2 \dd r 
		\bigg\lVert\int_{lh}^{(l+1)h} B(Y_l) ((l+1)h-s) \dd W_s
		\bigg\rVert_H^2}^{1/2} \\
	&\quad +\frac{h^2}{4\sqrt{3}} \mathbb{E} \bigg[  \int_0^1 
		\norm{F''(Y_l+r \cfuenf b(\cdot,Y_l)) - 2 F''(Y_l) 
		+ F''(Y_l-r \cfuenf b(\cdot,Y_l))}_{L^{(2)}(H,H)}^2  \dd r \\
	&\qquad \cdot \norm{B(Y_l)}_{HS(U_0,H)}^2 \bigg]^{1/2} . 
\end{align*}
In the next step, we use the Lipschitz continuity of both
$F'$ and $F''$ to compute
\begin{align*}
	\normLpH{2}{}{\Upsilon^F_l} &\leq C(K) h^2 {| \cdrei |} 
	\E{\norm{A Y_l +F(Y_l)}_H^4}^{1/2} \\
	&\quad +C(K) {| \cvier |} \E{\bigg\lVert\int_{lh}^{(l+1)h}
	\norm{b(\cdot,Y_l)}_H 
	B(Y_l) ((l+1)h-s) \dd W_s\bigg\rVert_H^2}^{1/2} \\
	&\quad + C(K) h^2 {| \cfuenf |} \E{\norm{b(\cdot,Y_l)}_H^2
	\norm{B(Y_l)}_{HS(U_0,H)}^2}^{1/2}. 
\end{align*}
By the Burkholder-Davis-Gundy-type inequality~\cite[Thm 4.37]{MR1207136}, 
it follows that
\begin{align*}
	&\E{\bigg\lVert\int_{lh}^{(l+1)h} \norm{b(\cdot,Y_l)}_H 
	B(Y_l) ((l+1)h-s) \dd W_s\bigg\rVert_H^2}^{1/2} \\
	&\leq \left( \frac{1}{3} \int_{lh}^{(l+1)h} \E{  \norm{b(\cdot,Y_l)}_H^2 
    \norm{B(Y_l)}_{HS(U_0,H)}^2} ((l+1)h-s)^2 \dd s \right)^{1/2} .
\end{align*}
Therefore, we get with the help of Lemma~\ref{Lma:Estimates_FB}, 
where $q=1$, that
\begin{equation}\label{eq:UpsF_fin}
        \normLpH{2}{}{\Upsilon^F_l} \leq C(K) \big( h^2{ |\cdrei| } 
        + h^{3/2}{ |\cvier| } + h^2 { |\cfuenf| } \big) \, .
\end{equation}

Now, we consider the stochastic integral in~\eqref{eq:Aprox_PsiF_PsiB}. 
By~\cite[Thm 4.37]{MR1207136}, we compute
\begin{align*}
	&\normALpH{\Big}{2}{}{\int_0^{mh} \sum_{l=0}^{m-1} 
	\ind_{[lh,(l+1)h)}(s) \, \e^{A(m-l-\tfrac{1}{2})h} \Upsilon^B_l(s) \dd W_s}^2 \\
	&\leq \int_0^{mh} \normALpHS{\Big}{2}{}{\sum_{l=0}^{m-1} 
	\ind_{[lh,(l+1)h)}(s) \, \e^{A(m-l-\tfrac{1}{2})h} \Upsilon^B_l(s)}^2 \dd s \\
	&\leq \sum_{l=0}^{m-1} \int_{lh}^{(l+1)h} \normA{\Big}{\e^{A(m-l-\tfrac{1}{2})h}}_{L(H)}^2
	\normLpHS{2}{}{\Upsilon^B_l(s)}^2 \dd s \, .
\end{align*}
With the triangle inequality, we get that
\begin{align*}
	&\normLpHS{2}{}{\Upsilon^B_l(s)}  \\
	&\leq (s-lh) \normALpHS{\bigg}{2}{}{\int_0^1 \Big[ B'(Y_l+r \cacht 
	[AY_l+F(Y_l)]) - B'(Y_l)\Big] \dd r \Big( AY_l+F(Y_l)\Big)}  \\
	&\quad +\normALpHS{\bigg}{2}{}{\int_0^1  \Big[ B'(Y_l+r \cneun b(\cdot,Y_l)) 
	- B'(Y_l)\Big] \dd r \bigg( \int_{lh}^s B(Y_l) \dd W_u \bigg)}  \\
	&\quad +\frac{1}{2} \bigg\lVert \int_0^1 \Big[ B''(Y_l+r \czehn b(\cdot,Y_l))
	-2 B''(Y_l) + B''(Y_l-r \czehn b(\cdot,Y_l))\Big] (1-r) \dd r  \\
	&\quad \quad \bigg( \int_{lh}^s B(Y_l) \dd W_u, \int_{lh}^s B(Y_l) \dd W_u \bigg) 
	\bigg\rVert_{L^2(\Omega;HS(U_0,H))} \\
	&\quad +\normALpHS{\bigg}{2}{}{{ 
	\int_0^1  B'(Y_l) \bigg( \int_{lh}^s \Big[ B'(Y_l+ r \celf b(\cdot,Y_l))
	- B'(Y_l)\Big] \bigg( \int_{lh}^u B(Y_l) \dd W_v \bigg) \dd W_u \bigg)  \dd r } } \\
	&\quad +\bigg\lVert \int_0^1 \Big[ B'\big( Y_l + r \tfrac{\aelf}{\celf} 
	\big[ b(\cdot,Y_l+\celf b(\cdot,Y_l))-b(\cdot,Y_l)\big]\big) - B'(Y_l) \Big] \dd r \\
	&\quad \quad \bigg( \int_{lh}^s \int_0^1 B'(Y_l+r\celf b(\cdot,Y_l)) \dd r
	\bigg( \int_{lh}^u B(Y_l) \dd W_v \bigg) \dd W_u \bigg)
	\bigg\rVert_{L^2(\Omega;HS(U_0,H))} \\
    &\leq (s-lh) \E{\int_0^1 \norm{B'(Y_l+r \cacht [AY_l+F(Y_l)]) 
    - B'(Y_l)}_{L(H,HS(U_0,H))}^2 \dd r \norm{AY_l+F(Y_l)}_H^2}^{1/2} \\
	&\quad +\E{\int_0^1  \norm{B'(Y_l+r \cneun b(\cdot,Y_l)) 
	- B'(Y_l)}_{L(H,HS(U_0,H))}^2 \dd r \Big\lVert \int_{lh}^s B(Y_l) \dd W_u
	\Big\rVert_H^2}^{1/2} \\
	&\quad +\frac{1}{2\sqrt{3}} \mathbb{E}\bigg[ \int_0^1 \norm{B''(Y_l + r \czehn
	b(\cdot,Y_l))-2 B''(Y_l) + B''(Y_l-r \czehn b(\cdot,Y_l) )
	}_{L^{(2)}(H,HS(U_0,H))}^2 \dd r  \\
	&\qquad \cdot \Big\lVert \int_{lh}^s B(Y_l) \dd W_u \Big\rVert_H^4 \bigg]^{1/2} \\
	&\quad + \int_0^1  \mathbb{E} \bigg[ \norm{ B'(Y_l) }_{L(H,HS(U_0,H))}^2 \\
	&\qquad \cdot \Big\lVert \int_{lh}^s
	\Big(B'(Y_l+r\celf b(\cdot,Y_l)) - B'(Y_l)\Big) \bigg( \int_{lh}^u B(Y_l) \dd W_v \bigg) 
	\dd W_u \Big\rVert_H^2 \bigg]^{1/2} \dd r \\
	&\quad +\mathbb{E} \bigg[ \int_0^1 \norm{B'\big( Y_l + r \tfrac{\aelf}{\celf} 
	\big[ b(\cdot,Y_l+\celf b(\cdot,Y_l))-b(\cdot,Y_l)\big]\big) 
	- B'(Y_l)}_{L(H,HS(U_0,H))}^2 \dd r \\
	&\qquad \cdot \Big\lVert \int_{lh}^s \int_0^1 B'(Y_l +r \celf b(\cdot,Y_l)) 
	\dd r \bigg( \int_{lh}^u B(Y_l) \dd W_v \bigg) \dd W_u
	\Big\rVert_H^2\bigg]^{1/2} \, .
\end{align*}
In the next step, we use the Lipschitz continuity 
of $B'$ and $B''$ as well as the boundedness of $B'$ to obtain
\begin{align*}
	&\normLpHS{2}{}{\Upsilon^B_l(s)} \\
	&\leq C(K) (s-lh) \, |\cacht| \, \E{\norm{AY_l+F(Y_l)}_H^4}^{1/2} \\
	&\quad + C(K) \, { |\cneun| } \, \E{\Big\lVert \int_{lh}^s
	\norm{b(\cdot,Y_l)}_H B(Y_l) \dd W_u \Big\rVert_H^2}^{1/2} \\
	&\quad + C(K) \, { |\czehn| } \, \E{\Big\lVert \int_{lh}^s
	\norm{b(\cdot,Y_l)}_H^{1/2} B(Y_l) \dd W_u \Big\rVert_H^4}^{1/2} \\
	&\quad + \int_0^1 \bigg( \int_{lh}^s  \E{ 
	\big\| B'(Y_l + r \celf b(\cdot,Y_l)) - B'(Y_l) \big\|_{L(H,HS(U_0,H))}^2
	\Big\| \int_{lh}^u B(Y_l) \dd W_v  \Big\|_H^2 } \dd u \bigg)^{1/2} \dd r \\
	&\quad + C(K) \, { \bigg| \frac{\aelf}{\celf} \bigg| } \,
	\mathbb{E} \bigg[ \Big\lVert \int_{lh}^s \int_0^1 B'(Y_l + r \, \celf \,
	b(\cdot,Y_l)) \dd r \\
	&\quad \quad \bigg( \int_{lh}^u \norm{b(\cdot,Y_l + \celf \, b(\cdot,Y_l)) 
	- b(\cdot,Y_l)}_H B(Y_l) \dd W_v \bigg) \dd W_u \Big\rVert_H^2 \bigg]^{1/2} .
\end{align*}

Now, equations~\eqref{eq:E_AYFY} and~\eqref{eq:E_bYBY} 
in Lemma~\ref{Lma:Estimates_FB} with $q=4$ 
and~\cite[Thm 4.37]{MR1207136} imply with the 
Lipschitz continuity of $B'$ that
\begin{align*}
	&\normLpHS{2}{}{\Upsilon^B_l(s)} \\
	&\leq C(K) (s-lh) { |\cacht| } \\
	&\quad + C(K) (s-lh)^{1/2} { |\cneun| } \E{\norm{b(\cdot,Y_l)}_H^2
	\norm{B(Y_l)}_{HS(U_0,H)}^2}^{1/2} \\
	&\quad + C(K) (s-lh) { |\czehn| } \E{\norm{b(\cdot,Y_l)}_H^2
	\norm{B(Y_l)}_{HS(U_0,H)}^4}^{1/2} \\
	&\quad + C(K)  
	\int_0^1 \bigg( \int_{lh}^s \E{ r^2 |\celf|^2 
	\Big\lVert \int_{lh}^u \norm{b(\cdot,Y_l)}_H B(Y_l) \dd W_v \Big\rVert_{H}^2 } 
	\dd u \bigg)^{1/2} \dd r \\
	&\quad +C(K) \bigg| \frac{\aelf}{\celf} \bigg| 
	\bigg( \int_{lh}^s \mathbb{E} \bigg[  \bigg( \int_0^1 \Big\lVert B'(Y_l + r \celf b(\cdot,Y_l))
	\Big\rVert_{L(H,HS(U_0,H))} \\
	&\qquad \cdot
	\Big\lVert \int_{lh}^u \norm{b(\cdot,Y_l + \celf b(\cdot,Y_l)) - b(\cdot,Y_l) }_H
	B(Y_l) \dd W_v \Big\rVert_{H} \dd r \bigg)^2 \dd u \bigg)^{1/2} .
\end{align*}
Then, we use~\eqref{eq:E_bYBY} in Lemma~\ref{Lma:Estimates_FB} 
(where $q\in \{2,4\}$) and the Lipschitz continuity of $b$ to 
calculate
\begin{align*}
	\normLpHS{2}{}{\Upsilon^B_l(s)} 
	&\leq C(K) \big( (s-lh) ( { |\cacht| + |\czehn| ) 
	+ (s-lh)^{1/2} |\cneun| } \big) \\
	&\quad + C(K) { |\celf| } \bigg(\int_{lh}^s \E{\Big\lVert 
	\int_{lh}^u \norm{b(\cdot,Y_l)}_H B(Y_l) \dd W_v 
	{ \Big\rVert_{H}^2 } } 
	\dd u \bigg)^{1/2} \\
	&\quad +C(K) { |\aelf| } \bigg(\int_{lh}^s \mathbb{E}\bigg[ 
	\Big\lVert \int_{lh}^u \norm{b(\cdot,Y_l))}_H B(Y_l) \dd W_v
	{ \Big\rVert_{H}^2 } \bigg] \dd u \bigg)^{1/2} .
\end{align*}
Another application of~\cite[Thm 4.37]{MR1207136} yields that
\begin{align*}
	&\normLpHS{2}{}{\Upsilon^B_l(s)} \\
	&\leq C(K) \big( (s-lh) ( { |\cacht| + |\czehn| ) 
	+ (s-lh)^{1/2} |\cneun| } \big) \\
	&\quad +C(K) ({ |\celf| + |\aelf| } ) \bigg(\int_{lh}^s \int_{lh}^u
	\E{\norm{b(\cdot,Y_l)}_H^2 {
	\norm{B(Y_l)}_{HS(U_0,H)}^2 } } \dd v \dd u \bigg)^{1/2} .
\end{align*}
Then, we use Lemma~\ref{Lma:Estimates_FB} again with $q=2$ and 
compute
\begin{equation}\label{eq:UpsB_fin}
	\normLpHS{2}{}{\Upsilon^B_l(s)} 
	\leq C(K) \big( (s-lh) ({ |\cacht| + |\czehn| + |\celf| + |\aelf| ) 
	+ (s-lh)^{1/2} |\cneun| } \big).
\end{equation}
As the final step, a combination of the estimates 
in~\eqref{eq:UpsF_fin} and~\eqref{eq:UpsB_fin} 
yields the inequality
\begin{align*}
	&\normLpH{2}{}{\bar{Y}_m-Y_m}^2 \\
	&\leq C(K) \bigg( \sum_{l=0}^{m-1} \normLpH{2}{}{\Upsilon_l^F} \bigg)^2 
	+ C(K) \sum_{l=0}^{m-1} \int_{lh}^{(l+1)h} \normLpHS{2}{}{\Upsilon_l^B (s)}^2
	\dd s \\
	&\leq C(K) \bigg( \sum_{l=0}^{m-1} h^2 { |\cdrei| } 
	+ h^{3/2}{ |\cvier| } + h^2 { |\cfuenf| } \bigg)^2 \\ 
	&\quad + C(K) \sum_{l=0}^{m-1} \int_{lh}^{(l+1)h} \big( (s-lh)^2
	({ |\cacht| + |\czehn| + |\celf| + |\aelf| } )^2 
	+ (s-lh) \cneun^2 \big) \dd s \Big) \\
    &\leq C(K) \Big( h^2 \big(\cdrei^2+\cfuenf^2+\cacht^2+\czehn^2+\celf^2
    +\aelf^2 \big) + h (\cvier^2+\cneun^2) \Big) . 
\end{align*}
Now, for coefficients $\cdrei,\cfuenf,\cacht,\czehn,\celf,\aelf\in\mathcal{O}(h^{\gamma-1})$ and $\cvier,\cneun\in\mathcal{O}(h^{\gamma-1/2})$, we obtain the final estimate
\begin{equation*}
    \normLpH{2}{}{\bar{Y}_m-Y_m} \leq C(K) h^{\gamma},
\end{equation*}
which completes the proof.
\begin{remark}\label{Rem:CoeffChoice}
	If we set the coefficients $\cdrei = hc_1, \cvier = hc_2, \cfuenf = \sqrt{h}c_3, 
	\cacht = hc_4, \cneun = hc_5, \czehn = \celf = \sqrt{h}c_6, \aelf = hc_7$ 
	for some arbitrary ${c_1},{c_2}, {c_3},{c_4}, {c_5}, {c_6},{c_7} \in \mathbb{R} 
	\setminus \{0\}$, we obtain the exponential stochastic Runge-Kutta type schemes 
	introduced in Section~\ref{Sec:DFM} with Table~\ref{SRK-scheme-Alg-Coeff}. It can 
	easily be seen that this choice results in a convergence of order $\gamma$ for 
	some $\gamma \in [1,\frac{3}{2})$.
\end{remark}
\end{proof}

We give the following two lemmas which are different from the statements 
given in~\cite{MR3534472} as the arguments may differ.
\begin{lem}\label{Lem:CC1}
	Let the setting described in Section~\ref{Sec:Setting} be fulfilled and let $U=U_0$. 
	If the specific second commutativity condition~\eqref{eq:comm_con_1}
	is satisfied, then
	\begin{align} \label{eq:comm_conv_rewrite_912_a}
	    \int_{lh}^{(l+1)h} B'(v) \bigg( \int_{lh}^s 
		B(\tilde{v}) \dd W_u \bigg) \dd W_s
		&= \frac{1}{2} B'(v) \Big( B(\tilde{v}) \Delta W_l \Big)
		\Delta W_l - \frac{h}{2} \sumJ B'(v) 
		\Big( B(\tilde{v}) g_j \Big) g_j
	\end{align}
	for all $v, \tilde{v} \in H$.
\end{lem}
\begin{proof}
	The proof follows as in~\cite{MR3320928}, where $v=\tilde{v}$.
\end{proof}
\begin{lem}\label{lem:CC2}
    Let the setting described in Section~\ref{Sec:Setting} be fulfilled 
    and let $U=U_0$. If the specific second commutativity 
    condition~\eqref{eq:comm_con_2}	is satisfied, then
	\begin{equation} \label{eq:comm_conv_rewrite_10111314_a}
		\begin{split}
		&\frac{1}{2} \int_{lh}^{(l+1)h} B''(w) \bigg( \int_{lh}^s 
		B(\tilde{w}) \dd W_u, \int_{lh}^s B(\hat{w}) 
		\dd W_u \bigg) \dd W_s \\ 
		&= \frac{1}{6} B''(w) \Big( B(\tilde{w}) \Delta W_l,
		B(\hat{w}) \Delta W_l \Big) \Delta W_l \\
		&\quad - \frac{1}{2} \sumJ B''(w) \Big( B(\tilde{w}) g_j,
		B(\hat{w}) g_j \Big) \WdsOK \\
		\end{split}
	\end{equation}
	and
	\begin{equation} \label{eq:comm_conv_rewrite_10111314_a-b}
		\begin{split}
			&\int_{lh}^{(l+1)h} B'(v) \bigg( \int_{lh}^s
			B'(\tilde{v}) \bigg( \int_{lh}^u B(\hat{v}) 
			\dd W_v \bigg) \dd W_u \bigg) \dd W_s \\
			&=
			\frac{1}{6} B'(v) \Big( B'(\tilde{v}) 
			\Big( B(\hat{v}) \Delta W_l \Big) \Delta W_l \Big) \Delta W_l 
			- \frac{h}{2} \sumJ B'(v) \Big( B'(\tilde{v}) 
			\Big( B(\hat{v}) g_j \Big) g_j \Big) \Delta W_l
		\end{split}
	\end{equation}
	for all $v, \tilde{v}, \hat{v}, w, \tilde{w}, \hat{w} \in H$.
\end{lem}
\begin{proof}
	This statement can be proved in the same way as Lemma~2 
	in~\cite{MR3320928} for arguments $v, \tilde{v}, \hat{v}, w, 
	\tilde{w}, \hat{w} \in H$ that are allowed to differ.
\end{proof}
\bibliographystyle{plainurl}
\bibliography{SPDE_all}

\begin{thebibliography}{10}

\bibitem{MR2996432}
A.~Barth and A.~Lang.
\newblock Milstein approximation for advection-diffusion equations driven by
  multiplicative noncontinuous martingale noises.
\newblock {\em Appl. Math. Optim.}, 66(3):387--413, 2012.
\newblock URL: \url{http://dx.doi.org/10.1007/s00245-012-9176-y}, \href
  {https://doi.org/10.1007/s00245-012-9176-y}
  {\path{doi:10.1007/s00245-012-9176-y}}.

\bibitem{MR3027891}
A.~Barth and A.~Lang.
\newblock {$L\sp p$} and almost sure convergence of a {M}ilstein scheme for
  stochastic partial differential equations.
\newblock {\em Stochastic Process. Appl.}, 123(5):1563--1587, 2013.
\newblock URL: \url{http://dx.doi.org/10.1016/j.spa.2013.01.003}, \href
  {https://doi.org/10.1016/j.spa.2013.01.003}
  {\path{doi:10.1016/j.spa.2013.01.003}}.

\bibitem{MR3534472}
S.~Becker, A.~Jentzen, and P.~E. Kloeden.
\newblock An exponential {W}agner-{P}laten type scheme for {SPDE}s.
\newblock {\em SIAM J. Numer. Anal.}, 54(4):2389--2426, 2016.
\newblock URL: \url{http://dx.doi.org/10.1137/15M1008762}, \href
  {https://doi.org/10.1137/15M1008762} {\path{doi:10.1137/15M1008762}}.

\bibitem{MR1207136}
G.~Da~Prato and J.~Zabczyk.
\newblock {\em Stochastic equations in infinite dimensions}, volume~44 of {\em
  Encyclopedia of Mathematics and its Applications}.
\newblock Cambridge University Press, Cambridge, 1992.
\newblock URL: \url{http://dx.doi.org/10.1017/CBO9780511666223}, \href
  {https://doi.org/10.1017/CBO9780511666223}
  {\path{doi:10.1017/CBO9780511666223}}.

\bibitem{MR3236753}
G.~Da~Prato and J.~Zabczyk.
\newblock {\em Stochastic equations in infinite dimensions}, volume 152 of {\em
  Encyclopedia of Mathematics and its Applications}.
\newblock Cambridge University Press, Cambridge, second edition, 2014.
\newblock URL: \url{http://dx.doi.org/10.1017/CBO9781107295513}, \href
  {https://doi.org/10.1017/CBO9781107295513}
  {\path{doi:10.1017/CBO9781107295513}}.

\bibitem{MR880357}
J.~Dixon and S.~McKee.
\newblock Weakly singular discrete {G}ronwall inequalities.
\newblock {\em Z. Angew. Math. Mech.}, 66(11):535--544, 1986.
\newblock \href {https://doi.org/10.1002/zamm.19860661107}
  {\path{doi:10.1002/zamm.19860661107}}.

\bibitem{MR4112639new}
C.~{\noopsort{Hallern}}{von Hallern} and A.~R{\"o}{\ss}ler.
\newblock An analysis of the {M}ilstein scheme for {SPDE}s without a
  commutative noise condition.
\newblock In {\em Monte {C}arlo and quasi-{M}onte {C}arlo methods}, volume 324
  of {\em Springer Proc. Math. Stat.}, pages 503--521. Springer, Cham, 2020.
\newblock \href {https://doi.org/10.1007/978-3-030-43465-6_25}
  {\path{doi:10.1007/978-3-030-43465-6_25}}.

\bibitem{HalRoe2023pre}
C.~{\noopsort{Hallern}}{von Hallern} and A.~R\"{o}{\ss}ler.
\newblock {A derivative-free Milstein type approximation method for SPDEs
  covering the non-commutative noise case.}
\newblock {\em Stoch. PDE: Anal. Comp.}, 2022.
\newblock \href {https://doi.org/10.1007/s40072-022-00274-6}
  {\path{doi:10.1007/s40072-022-00274-6}}.

\bibitem{MR3320928}
A.~Jentzen and M.~R{\"o}ckner.
\newblock A {M}ilstein scheme for {SPDE}s.
\newblock {\em Found. Comput. Math.}, 15(2):313--362, 2015.
\newblock URL: \url{http://dx.doi.org/10.1007/s10208-015-9247-y}, \href
  {https://doi.org/10.1007/s10208-015-9247-y}
  {\path{doi:10.1007/s10208-015-9247-y}}.

\bibitem{MR1825100}
P.~E. Kloeden and S.~Shott.
\newblock Linear-implicit strong schemes for {I}t\^o-{G}alerkin approximations
  of stochastic {PDE}s.
\newblock {\em J. Appl. Math. Stochastic Anal.}, 14(1):47--53, 2001.
\newblock URL: \url{http://dx.doi.org/10.1155/S1048953301000053}, \href
  {https://doi.org/10.1155/S1048953301000053}
  {\path{doi:10.1155/S1048953301000053}}.

\bibitem{MR3842926}
C.~Leonhard and A.~R\"{o}{\ss}ler.
\newblock Enhancing the {O}rder of the {M}ilstein {S}cheme for {S}tochastic
  {P}artial {D}ifferential {E}quations with {C}ommutative {N}oise.
\newblock {\em SIAM J. Numer. Anal.}, 56(4):2585--2622, 2018.
\newblock \href {https://doi.org/10.1137/16M1094087}
  {\path{doi:10.1137/16M1094087}}.

\bibitem{MR3949104}
C.~Leonhard and A.~R\"{o}{\ss}ler.
\newblock Iterated stochastic integrals in infinite dimensions: approximation
  and error estimates.
\newblock {\em Stoch. Partial Differ. Equ. Anal. Comput.}, 7(2):209--239, 2019.
\newblock \href {https://doi.org/10.1007/s40072-018-0126-9}
  {\path{doi:10.1007/s40072-018-0126-9}}.

\bibitem{MR3047942}
G.~J. Lord and A.~Tambue.
\newblock Stochastic exponential integrators for the finite element
  discretization of {SPDE}s for multiplicative and additive noise.
\newblock {\em IMA J. Numer. Anal.}, 33(2):515--543, 2013.
\newblock URL: \url{http://dx.doi.org/10.1093/imanum/drr059}, \href
  {https://doi.org/10.1093/imanum/drr059} {\path{doi:10.1093/imanum/drr059}}.

\bibitem{MR4032895}
C.~Reisinger and Z.~Wang.
\newblock Stability and error analysis of an implicit {M}ilstein finite
  difference scheme for a two-dimensional {Z}akai {SPDE}.
\newblock {\em BIT}, 59(4):987--1029, 2019.
\newblock \href {https://doi.org/10.1007/s10543-019-00761-8}
  {\path{doi:10.1007/s10543-019-00761-8}}.

\bibitem{MR3011387}
X.~Wang and S.~Gan.
\newblock A {R}unge-{K}utta type scheme for nonlinear stochastic partial
  differential equations with multiplicative trace class noise.
\newblock {\em Numer. Algorithms}, 62(2):193--223, 2013.
\newblock URL: \url{http://dx.doi.org/10.1007/s11075-012-9568-8}, \href
  {https://doi.org/10.1007/s11075-012-9568-8}
  {\path{doi:10.1007/s11075-012-9568-8}}.

\end{thebibliography}
\end{document}